\documentclass{article}

\usepackage[utf8]{inputenc}
\usepackage{amssymb,amstext,amsmath,amscd,amsthm,amsfonts,enumerate,latexsym}
\usepackage{color}
\usepackage{stmaryrd }

\usepackage{relsize}
\usepackage[bbgreekl]{mathbbol}
\usepackage{amsfonts}
\DeclareSymbolFontAlphabet{\mathbb}{AMSb} 
\DeclareSymbolFontAlphabet{\mathbbl}{bbold}
\newcommand{\Prism}{{\mathlarger{\mathbbl{\Delta}}}}

\usepackage{titlesec}
\titleformat{\section}{\centering\large\sc}{\thesection.}{0.7em}{}
\usepackage{etoolbox}
\patchcmd{\abstract}
  {\bfseries\abstractname}
 {\scshape\abstractname}
  {}{}

\usepackage{hyperref}
\hypersetup{
    colorlinks=true,
    linkcolor=blue,
    citecolor=blue,
    filecolor=blue,      
    urlcolor=blue,
}
\usepackage[backend=biber,style=alphabetic,useprefix=true,autopunct=true,sorting=nyt,maxbibnames=99,maxalphanames=99]{biblatex}

\newcommand{\sortas}[1]{}
    \usepackage{afterpage}
\usepackage{lipsum}

\newcommand\blfootnote[1]{%
  \begingroup
  \renewcommand\thefootnote{}\footnote{#1}%
  \addtocounter{footnote}{-1}%
  \endgroup
}

\usepackage[all]{xy}
\usepackage{tikz-cd}
\begin{document}
\theoremstyle{plain}
\newtheorem{thm}{Theorem}[section]
\newtheorem{theorem}[thm]{Theorem}

\numberwithin{equation}{thm}
\newtheorem*{thm*}{Theorem}
\newtheorem*{cor*}{Corollary}
\newtheorem*{mthm}{Main Theorem}
\newtheorem*{mcor}{Theorem \ref{maincor}}
\newtheorem{thms}[thm]{Theorems}
\newtheorem{prop}[thm]{Proposition}
\newtheorem{proposition}[thm]{Proposition}
\newtheorem{prodef}[thm]{Proposition and Definition}
\newtheorem{lemma}[thm]{Lemma}
\newtheorem{lem}[thm]{Lemma}
\newtheorem{cor}[thm]{Corollary}

\newtheorem{corollary}[thm]{Corollary}
\newtheorem{claim}[thm]{Claim}
\newtheorem*{claim*}{Claim}

\theoremstyle{definition}
\newtheorem{defn}[thm]{Definition}
\newtheorem{hypothesis}[thm]{Hypothesis}
\newtheorem{defnProp}[thm]{Definition and Proposition}
\newtheorem{definition}[thm]{Definition}
\newtheorem{ex}[thm]{Example}
\newtheorem{exercise}[thm]{Exercise}
\newtheorem{Example}[thm]{Example}
\newtheorem{example}[thm]{Example}
\newtheorem*{acknowledgements}{Acknowledgements}

\newtheorem{notation}[thm]{Notation}
\newtheorem{rem}[thm]{Remark}
\newtheorem{remark}[thm]{Remark}

\theoremstyle{remark}


\def\D{\mathcal{D}}

\def\E{\operatorname{E}}
\def\m{\mathfrak m}

\def\J{\mathrm{J}}

\newcommand{\M}{\mathcal{M}}
\newcommand{\N}{\mathcal{N}}

\newcommand{\fkS}{\mathfrak{S}}

\def\Sp{\operatorname{Spec}}
\def\Spf{\operatorname{Spf}}
\def\Spa{\operatorname{Spa}}

\def\A{{\mathcal A}}
\def\B{{\mathcal B}}
\def\F{{\mathcal F}}
\def\M{{\mathcal M}}
\def\N{{\mathcal N}}
\def\O{{\mathcal O}}
\def\H{{\mathcal H}}
\def\G{{\mathcal G}}
\def\R{{\mathcalR}}
\def\I{{\mathcal I}}
\def\fM{\mathfrak M}
\def\fN{\mathfrak N}
\def\fF{\mathfrak F}
\def\T{\mathcal T}
\def\fl{\pi^\flat}
\def\E{{\mathcal E}}
\renewcommand{\R}{\mathcal{R}}


\newcommand{\mf}{\mathrm{MF}^{\varphi_q}}
\newcommand{\hp}{\mathrm{HP}^{\varphi_q}}
\newcommand{\vect}{\mathrm{Vect}^{\varphi_q}}
\title{\large\bf PRISMATIC F-CRYSTALS AND $E$-CRYSTALLINE\\ GALOIS REPRESENTATIONS}
\author{\normalsize DAT PHAM}
\date{}
\maketitle
\begin{abstract}
Let $K$ be a complete discretely valued field of mixed characteristic $(0,p)$ with perfect residue field, and let $E$ be a finite extension of $\mathbf{Q}_p$ contained in $K$. We show that the category of prismatic $F$-crystals on $\O_K$ (relative to $E$ in a suitable sense) is equivalent to the category of $\O_E$-lattices in $E$-crystalline $G_K$-representations introduced by Kisin--Ren, extending a previous result of Bhatt--Scholze in the case $E=\mathbf{Q}_p$. As a key ingredient in the proof, by adapting a lemma of Du--Liu we prove  a general full faithfulness result for certain vector bundles on the prismatic site, which simplifies and refines the key descent step in the approach of Bhatt--Scholze without invoking the Beilinson fibre sequence. 
\end{abstract}


\setcounter{tocdepth}{1}

\tableofcontents
\section{Introduction}
Let $K/\mathbf{Q}_p$ be a completed discrete valued extension with perfect residue field $k$ of characteristic $p>0$, fixed completed algebraic closure $C$, and absolute Galois group $G_K$. An important aspect of integral $p$-adic Hodge theory is the study of lattices in crystalline (or more generally, semistable) $G_K$-representations. There have been various (partial) classifications of such lattices, including Fontaine--Laffaille's theory \cite{FontaineLaffaille}, Breuil's theory of strongly divisible $S$-lattices \cite{breuil02}, and Kisin's theory of Breuil--Kisin modules \cite{Kis06}. In \cite{prismatic}, Bhatt and Scholze give a site-theoretic description of such lattices, which has the nice feature that it can recover many of the previous classifications by ``evaluating'' suitably. To recall their result, let $(\O_K)_{\Prism}$ denote the absolute prismatic site of $\O_K$; this comes equipped with a structure sheaf $\O_{\Prism}$, a ``Frobenius'' $\varphi: \O_{\Prism}\to \O_{\Prism}$, and an ideal sheaf $\I_{\Prism}\subseteq \O_{\Prism}$.
\begin{defn}
A prismatic $F$-crystal on $\O_K$ is a crystal of vector bundles on the ringed site $((\O_K)_{\Prism},\O_{\Prism})$ equipped with an isomorphism $(\varphi^*\E)[1/\I_{\Prism}]\simeq \E[1/\I_{\Prism}]$; denote by $\mathrm{Vect}^{\varphi}((\O_K)_{\Prism},\O_{\Prism})$ the resulting category. Similarly, we obtain the category $\mathrm{Vect}^{\varphi}((\O_K)_{\Prism},\O_{\Prism}[1/\I_{\Prism}]^{\wedge}_p)$ of so-called Laurent $F$-crystals. 
\end{defn}
In the statement below, $\mathrm{Rep}_{\mathbf{Z}_p}(G_K)$ denotes the category of finite free $\mathbf{Z}_p$-modules $T$ equipped with a continuous $G_K$-action, and $\mathrm{Rep}_{\mathbf{Z}_p}^{\mathrm{cris}}(G_K)$ denotes the subcategory spanned by those $T$ for which $T[1/p]$ is crystalline. 
\begin{thm}[{\cite{prismatic}\footnote{The bottom horizontal equivalence was obtained independently by Z. Wu \cite{prismaticWu}.}}]\label{main thm BS}
There is a commutative diagram 
\begin{displaymath}
    \begin{tikzcd}
    \mathrm{Vect}^{\varphi}((\O_K)_{\Prism},\O_{\Prism})\ar[d,hook]\ar[r,"\simeq"] & \mathrm{Rep}_{\mathbf{Z}_p}^{\mathrm{cris}}(G_K)\ar[d,hook]\\
    \mathrm{Vect}^{\varphi}((\O_K)_{\Prism},\O_{\Prism}[1/\I_{\Prism}]^{\wedge}_p)\ar[r,"\simeq"] & \mathrm{Rep}_{\mathbf{Z}_p}(G_K).
    \end{tikzcd}
\end{displaymath}
Here the vertical embeddings are given by the obvious maps; the horizontal equivalences are induced by evaluating on the Fontaine's prism ${A}_{\inf}$,the so-called \' etale realization functor.
\end{thm}
Motivated by our studies of the stacks of Lubin--Tate $(\varphi,\Gamma)$-modules in \cite{pham2023moduli}, it is natural to ask if there is a variant of Theorem \ref{main thm BS} for coefficient rings other than $\mathbf{Z}_p$. More specifically, let $E$ be a finite extension of $\mathbf{Q}_p$ with residue field $\mathbf{F}_q$ and a fixed uniformizer $\pi$; we are interested in crystalline representations of $G_K$ on finite dimensional $E$-vector spaces (or rather, $G_K$-stable $\O_E$-lattices in such). 
\begin{hypothesis}
We assume throughout that there is an embedding $\tau_0: E\hookrightarrow K$, which we will fix once and for all.
\end{hypothesis} 
\begin{defn}[{\cite{KisinRen}}]
An $E$-linear representation $V$ of $G_K$ is called $E$-crystalline if it is crystalline (as a $\mathbf{Q}_p$-representation), and the $C$-semilinear representation $\bigoplus_{\tau\ne \tau_0}V\otimes_{E,\tau}C$ is trivial\footnote{This is not quite the original definition in \cite{KisinRen}, but can be easily seen to be equivalent to it (see Lemma \ref{analytic rep} below).}.
\end{defn}
A natural source of such representations comes from the rational Tate modules of $\pi$-divisible $\O_E$-modules over $\O_K$ (see Lemma \ref{pi divisible and analytic}). Moreover, just as in the case $E=\mathbf{Q}_p$, $E$-crystalline representations can be classified using weakly admissible filtered $\varphi_q$-modules over $K$. In fact, the above notion is indeed a natural extension of the usual notion in the sense that there is a natural period ring $B_{\mathrm{cris},E}$ with the property that an $E$-linear representation $V$ is $E$-crystalline if and only if $V\otimes_E B_{\mathrm{cris},E}$ is trivial as a $B_{\mathrm{cris},E}$-semilinear representation (cf. Theorem \ref{characterization of E-crystalline main body}).

Using the theory of Lubin--Tate $(\varphi,\Gamma)$-modules, Kisin--Ren give a classification of the category of Galois stable lattices in $E$-crystalline representations of $G_K$ (under a condition on the ramification of $K$) \cite[Theorem (0.1)]{KisinRen}, generalizing the earlier classification in terms of Wach modules by Wach, Colmez and Berger (cf. \cite{bergerlimitcrys}). 

In another direction, in \cite{marks2023prismatic} Marks defines a variant of the (absolute) prismatic site $(\O_K)_{\Prism,\O_E}$ of $\O_K$ ``relative to $\O_E$'', using the notion of $\O_E$-prisms\footnote{These are called $E$-typical prisms in \cite{marks2023prismatic}.}, a mild generalization of prisms: they are roughly $\O_E$-algebras $A$ equipped with a lift $\varphi_q: A\to A$ of the $q$-power Frobenius modulo $\pi$ together with a Cartier divisor $I$ of $\Sp(A)$ satisfying $\pi\in (I,\varphi_q(I))$. It is shown in \textit{loc. cit.} that the \' etale realization functor again defines an equivalence
\begin{align*}
     T: \vect((\O_K)_{\Prism,\O_E},\O_{\Prism}[1/\I_{\Prism}]^{\wedge}_p) &\simeq \mathrm{Rep}_{\O_E}(G_K).
\end{align*}
In this paper, we push this analogy further by showing the following extension of Theorem \ref{main thm BS}.
\begin{thm}[{Theorem \ref{main thm body}}]\label{main thm intro}
The \' etale realization functor ind
uces an equivalence 
\begin{displaymath}
        T: \vect((\O_K)_{\Prism,\O_E},\O_{\Prism})\simeq \mathrm{Rep}^{\mathrm{cris}}_{\O_E}(G_K),
    \end{displaymath}
where the target denotes the category of Galois stable $\O_E$-lattices in $E$-crystalline representations of $G_K$.
\end{thm}
As will be explained in \textsection \ref{relation kisin--ren} below, by evaluating at a suitable prism in $(\O_K)_{\Prism}$, Theorem \ref{main thm intro} encodes the main equivalence from Kisin--Ren's work \cite{KisinRen}.

Finally, by combining Theorem \ref{main thm intro} with a key result from \cite{prismaticdieudonne} (generalized to the ``$\O_E$-context'' in \cite{ito2023prismatic}), we deduce the following classification of $\pi$-divisible $\O_E$-modules over $\O_K$.

\begin{theorem}[{Theorem \ref{classification pi divisible Cheng}}]
There is a natural equivalence between the category of $\pi$-divisible $\O_E$-modules over $\O_K$ and the category of minuscule Breuil--Kisin modules over $\fkS_E$.
\end{theorem}
\subsection{Sketch of the proof of Theorem \texorpdfstring{\ref{main thm intro}}{1.5}}
As alluded to above, an important observation is that the condition of being $E$-crystalline can be characterized in a manner similar to the usual notion for $E=\mathbf{Q}_p$. Namely, there is a natural period ring $B_{\mathrm{cris},E}$ with the property that an $E$-linear representation $V$ is $E$-crystalline if and only if $V\otimes_E B_{\mathrm{cris},E}$ is trivial as a $B_{\mathrm{cris},E}$-representation; see Theorem \ref{characterization of E-crystalline main body}. Once this is justified, that the \' etale realization functor is well-defined and fully faithful can be proved in exactly the same way as in \cite{prismatic}. For essential surjectivity, we again follow largely the proof in \cite{prismatic}; the main difference here is that instead of invoking the Beilinson fibre sequence from \cite{AMMNBeilinson} for the key descent step, we are able to prove the following more general result by adapting a key lemma from \cite{duliu}. 
\begin{prop}[{Theorem \ref{boundedness descent data}}]\label{nice statement}
Let $(A,(d))$ be a transversal $\O_E$-prism. Then the base change
\begin{displaymath}
    \mathrm{Vect}^{\varphi_q}(A)[1/\pi]\to \mathrm{Vect}^{\varphi_q}(A\langle d/\pi\rangle [1/\pi])
\end{displaymath}
is fully faithful; here the source denotes the isogeny category of $\vect(A)$.
\end{prop}
We regard Proposition \ref{nice statement} as a result of independent interest. For instance, as mentioned above, by specializing to the prism $A=\Prism_{\O_C\widehat{\otimes}_{\O_K}\O_C}$, this recovers (and refines) Proposition 6.10 in \cite{prismatic}. Furthermore, by specializing to a Breuil--Kisin prism $(\fkS,I)$, this recovers the embedding 
\begin{displaymath}
    \mathrm{Vect}^{(\varphi,N)}(\fkS)[1/p]\hookrightarrow \mathrm{Vect}^{(\varphi,N)}(\O)
\end{displaymath}
from \cite[Lemma 1.3.13]{Kis06} without using Kedlaya's results on slope filtrations; here $\O$ denotes the ring of functions on the rigid open unit disk over $K_0$.

The proof of Proposition \ref{nice statement} proceeds by first reducing to the case of finite free modules. In this case, by working with matrices for the $\varphi_q$-actions, we reduce to showing that if
\begin{displaymath}
 d^hY=B\varphi_q(Y)C   
\end{displaymath} 
with $h\geq 0$, $Y\in M_d(A\langle d/\pi\rangle)$ and $B, C\in M_d(A)$, then $Y\in M_d(A[1/\pi])$. Here the idea (due to Du--Liu) is to approximate $d$-adically $Y$ by matrices in $M_d(A)$. This is possible thanks to the following variant of \cite[Lemma 2.2.10]{duliu} on the contracting effect of the Frobenius on the $d$-adic filtration on $A\langle d/\pi\rangle$. 
\begin{lem}[Lemma \ref{lemma phi(Fil) Du-Liu}]\label{lemma phi(Fil) Du-Liu}
    Let $(A,(d))$ be a transversal $\O_E$-prism. Then given any $h\geq 0$, 
    \begin{displaymath}
        \varphi_q(d^m A\langle d/\pi\rangle)\subseteq A+d^{m+h}A\langle d/\pi\rangle
    \end{displaymath}
    for all $m\gg 0$ (depending only on $h$).
\end{lem}

We now detail the organization of the chapter. In Section \ref{section analytic reps}, we recall the definition of $E$-crystalline representations from \cite{KisinRen}, and then in the appendix, following \cite[Chapitre 10]{FFCurves}, we interpret it in terms of vector bundles on the Fargues--Fontaine curve (Proposition \ref{equivalence E crystalline}). In particular, we show that the category of $E$-crystalline representations of $G_K$ is equivalent to the category of weakly admissible filtered $\varphi_q$-modules over $K$, and moreover that being $E$-crystalline is equivalent to being $B_{\mathrm{cris},E}$-admissible for a natural period ring $B_{\mathrm{cris},E}$. In Section \ref{BK theory}, we adapt some constructions from \cite{Kis06} to the present context. Next, in Section \ref{prismatic story}, we review briefly the notion of $\O_E$-prisms and define the \' etale realization functor in Theorem \ref{main thm BS}. Full faithfulness is then addressed in Subsection \ref{full faithfulness of etale realization}. Subsection \ref{essential surjectivity} begins with some further ring-theoretic properties of transversal prisms, culminating with the proof of Proposition \ref{boundedness descent data}, which is then used in the proof of essential surjectivity. Finally, in the last two subsections, we briefly discuss an application of Theorem \ref{main thm intro} to the theory of $\pi$-divisible $\O_E$-modules over $\O_K$ as well as its relation with Kisin--Ren's classification in \cite{KisinRen}.     
\begin{acknowledgements}
I would like to thank my advisors Bao Le Hung and Stefano Morra for their continuous support during the preparation of this paper. I would also like to thank Tong Liu for making me aware of the paper \cite{CaisLiu}. Finally, I am grateful to the referee for their helpful suggestions which help improve
the readability of the paper.

This project has received funding from the European Union’s Horizon 2020 research and innovation programme under the Marie Sk\l odowska-Curie grant agreement No. 945322.
\end{acknowledgements}
\section{\texorpdfstring{$E$}{E}-crystalline Galois representations}\label{section analytic reps}
\subsection{Definitions}
Recall again that we have fixed throughout an embedding $\tau_0: E\hookrightarrow K$; a general embedding $E\hookrightarrow C$ will be typically denoted by $\tau$. Let $V$ be an $E$-linear representation of $G_K$ which is crystalline (in the usual sense, i.e., as a $\mathbf{Q}_p$-linear representation). Then 
\begin{displaymath}
D_{\mathrm{dR}}(V):=(V\otimes_{\mathbf{Q}_p}B_{\mathrm{dR}})^{G_K}
\end{displaymath}
is naturally a finite free $E\otimes_{\mathbf{Q}_p}K$-module equipped with a (finite, separated, exhausted) filtration by $E\otimes_{\mathbf{Q}_p}K$-submodules (whose associated graded pieces are finite projective, but not necessarily of constant rank, or equivalently, free). In particular, there is a decomposition 
\begin{align}\label{decomposition bdr}
    D_{\mathrm{dR}}(V)=\prod_{\m}D_{\mathrm{dR}}(V)_{\m}
\end{align}
where $\m$ runs over the (finite) set of maximal ideals in $E\otimes_{\mathbf{Q}_p}K$.

The following definition was given by Kisin--Ren \cite{KisinRen}.
\begin{defn}\label{analytic definition}
$V$ is called $E$-crystalline if (it is crystalline and) the induced filtration on $D_{\mathrm{dR}}(V)_{\m}$ is trivial for all $\m\ne \m_0$, where $\m_0$ denotes the maximal ideal corresponding to the multiplication map $K\otimes_{\mathbf{Q}_p}E\twoheadrightarrow K$ defined by the embedding $\tau_0$.
\end{defn}
\begin{lem}\label{analytic rep}
Let $V$ be a crystalline $E$-linear representation of $G_K$. Then $V$ is $E$-crystalline if and only if $\bigoplus_{\tau\ne \tau_0}V\otimes_{E,\tau}C$ is trivial as a $C$-semilinear representation of $G_K$\footnote{This latter condition is called ``$E$-analytic'' in many references (see e.g. \cite{Bergeranalytic}); we have decided to use the original term ``$E$-crystalline'' of Kisin--Ren to avoid potential confusion.}.
\end{lem}
Here the action of $g\in G_K$ on $\bigoplus_{\tau\ne \tau_0}V\otimes_{E,\tau}C$ is given by the maps $1\otimes g: V\otimes_{E,\tau}C\to V\otimes_{E,g\tau}C$. (Note that the diagonal action of $G_K$ on $V\otimes_{E,\tau}C$ is not well-defined in general: the map $\tau: E\hookrightarrow C$ is not $G_K$-equivariant unless $\tau(E)\subseteq K$.) 
\begin{proof}
Recall that the filtration on $D_{\mathrm{dR}}(V)$ satisfies $\mathrm{gr}^i D_{\mathrm{dR}}(V)\simeq (V\otimes_{\mathbf{Q}_p}C(i))^{G_K}$ for each $i$. Thus using the decomposition $V\otimes_{\mathbf{Q}_p}C=\prod_{\tau}V\otimes_{E,\tau}C$, one can check that $\mathrm{gr}^i D_{\mathrm{dR}}(V)_{\m}=0$ for all $i\ne 0$ and $\m\ne \m_0$ if and only if the $C$-semilinear representation $W:=\bigoplus_{\tau\ne \tau_0}V\otimes_{E,\tau}C$ satisfies $(W\otimes_{C}C(i))^{G_K}=0$ for all $i\ne 0$. As $W$ is Hodge--Tate (being a quotient of the Hodge--Tate representation $V\otimes_{\mathbf{Q}_p}C$), this amounts precisely to saying that $W$ is trivial. 
\end{proof}
\subsection{Relation with filtered isocrystals}
For our purpose of proving Theorem \ref{main thm intro}, the following equivalent characterizations of the category of $E$-crystalline Galois representations will be fundamental. Before stating the result, we introduce the crysalline period ring in our context.
\begin{notation}
Let $B_{\mathrm{cris},E}$ denote Fontaine's crystalline period ring defined using $E$ and $\tau_0: E\hookrightarrow K\subseteq C$. More precisely, let ${A}_{\inf,E}:=W_{\O_E}(\O_C^{\flat})$ (defined using the embedding $\tau_0$), and let $A_{\mathrm{cris},E}$ be the $\pi$-completed $\O_E$-PD envelope of $A_{\inf,E}$ with respect to the kernel of the Fontaine's map $\theta_E: A_{\inf,E}\twoheadrightarrow \O_{C}$, i.e. $A_{\mathrm{cris},E}$ is the $\pi$-adic completion of the subring
\begin{displaymath}
    A_{\inf,E}[\xi^{q^n}/\pi^{1+q+\ldots+q^{n-1}},n\geq 1]\subseteq A_{\inf,E}[1/\pi],
\end{displaymath}
where $\xi$ is one (or, any) generator of $\ker(\theta_E)$. We then let $B_{\mathrm{cris},E}^+:=A_{\mathrm{cris},E}[1/\pi]$ and $B_{\mathrm{cris},E}:=B_{\mathrm{cris},E}^+[1/t_E]$\footnote{In \cite{KisinRen}, ${B}_{\mathrm{cris},E}$ denotes the ring $B_{\mathrm{cris}}\otimes_{E_0}E$; these are in general different rings.}. These will play the roles of the usual crystalline period rings in the story over $\mathbf{Q}_p$. (As the notation suggests $t_E$ is the analogue of the usual element $t=\log([\epsilon])$ (the ``$2\pi i$ of Fontaine'') in our ``$\O_E$-context''; see Appendix \ref{appendix crystalline} for more details.)    
\end{notation}
\begin{thm}[Remark \ref{analytic crys natural notion}, Theorem \ref{weakly admissible vs crys rep}]\label{characterization of E-crystalline main body}
\begin{itemize}
    \item[\emph{(1)}] Let $V\in \mathrm{Rep}_E(G_K)$. Then $V$ is $E$-crystalline if and only if $V\otimes_{E}B_{\mathrm{cris},E}$ is trivial as a $B_{\mathrm{cris},E}$-semilinear representation of $G_K$.
    \item[\emph{(2)}] The functor 
\begin{align*}
    D_{\mathrm{cris},E}: V & \mapsto (V\otimes_E B_{\mathrm{cris},E})^{G_K}
\end{align*}
defines an equivalence from the category of $E$-crystalline representations of $G_K$ onto the category of weakly admissible filtered $\varphi_q$-modules.
\end{itemize}  
\end{thm}
(We refer the reader to Appendix \ref{appendix crystalline} for the notion of (weakly admissible) filtered $\varphi_q$-modules, which is simply a straightforward extension of the usual notion.)

Our proof of Theorem \ref{characterization of E-crystalline main body} rests on the realization of the relevant objects as certain vector bundles on the Fargues--Fontaine curve (associated to $E$ and the embedding $\tau_0$). Since this will require a digression on the Fargues--Fontaine curve which has little do to with the rest of the chapter, we defer the proof to Appendix \ref{appendix crystalline}. 
\begin{remark}
    The above equivalence between $E$-crystalline representations of $G_K$ and weakly admissible filtered $\varphi_q$-modules is presumably standard, although we cannot find a reference which explicitly states it. We can deduce it more directly (as an abstract equivalence) by combining the usual equivalence for $E=\mathbf{Q}_p$ with a standard passage from $\varphi$-modules to $\varphi_q$-modules (cf. \cite[\textsection 3.3]{KisinRen}). Here we prefer the more geometric perspective via the Fargues--Fontaine curve as it gives an explicit recipe for the equivalence as above, and also makes the analogy with the usual case $E=\mathbf{Q}_p$ more transparent\footnote{It also seems to suggest the idea that in order to treat all $E$-linear crystalline (i.e. not necessarily $E$-crystalline) representations of $G_K$, one should consider modifications of vector bundles on the Fargues--Fontaine curve $X_E$ at finitely many points (rather than just one point as in the case $E=\mathbf{Q}_p$).}. 
\end{remark}
\section{Theory of Breuil--Kisin modules}\label{BK theory}
In this section, we adapt some constructions from \cite{Kis06} to the ``$\O_E$-context''. We note that in \cite{CaisLiu}, Cais and Liu have extended a large part of the theory in \cite{Kis06} to accommodate more general coefficient rings and lifts of Frobenius. In particular, we do not claim originality in this part. However, in our present context (corresponding to the Frobenius lift $f(u)=u^q$) much of the discussion in \cite{CaisLiu} simplifies, and we can present the material largely in parallel with \cite{Kis06}.

\subsection{Preliminaries}
We fix once and for all a uniformizer $\pi_K\in K$. Let $E(u)\in W_{\O_E}(k)[u]$ be the Eisenstein polynomial of $\pi_K$ over $K_{0,E}$. Let $\fkS_E:=W_{\O_E}(k)[[u]]$ and note $\varphi_q: \fkS_E\to \fkS_E$ the ring map that extends the Frobenius in $W_{\O_E}(k)$ and sends $u$ to $u^q$. The map $\tilde{\theta}: \fkS_E\twoheadrightarrow \O_K, u\mapsto \pi_K$ is surjective with kernel $I=(E(u))$. 

Let $\Delta:=\Spa(\fkS_E)-\{\pi=0\}$ be the adic generic fiber of $\mathrm{Spa}(\fkS_E)$ over $K_{0,E}$; this is the adic open unit disk over $K_{0,E}$. Let $\O$ be the ring of functions on $X$. As $\Delta$ is an increasing union of the rational opens $U_n:=\{|\varphi_q^n(E(u))|\leq |\pi|\ne 0\}=\{|u^{eq^n}|\leq |\pi|\ne 0\}$ for $n\geq 0$, we have
\begin{align}\label{functions open disk inverse limit}
\O &=\varprojlim_n\fkS_E\left\langle \dfrac{\varphi_q^n(E(u))}{\pi}\right\rangle [1/\pi]\\
&=\bigcap_{n\geq 0}K_{0,E}\left\langle \dfrac{u^{e q^n}}{\pi}\right\rangle\;\text{inside $K_{0,E}[[u]]$}.
\end{align}
Recall that by the work of Lazard \cite{Lazard} that each $K_{0,E}\langle u^{eq^n}/\pi\rangle$ is a PID and $\O$ is a Bézout domain. In particular, finite projective modules over $\O$ are free (see e.g. \cite[Proposition 2.5]{kedlaya}). Moreover, base change defines an equivalence
\begin{displaymath}
    \mathrm{Vect}(\O)\simeq \lim_n\mathrm{Vect}(\fkS_E\langle u^{e q^n}/\pi\rangle [1/\pi]).
\end{displaymath}
Denote again by $\varphi_q: \O\to \O$ the map induced by $\varphi_q$ on $\fkS_E$. For $n\geq 0$, let $x_n\in \Delta$ be the unique point where $\varphi_q^n(E(u))$ vanishes, i.e. $x_n: \fkS_E\to \fkS_E/\varphi_q^n(E(u))[1/\pi]=:K_n$. Let $\widehat{\fkS}_n$ denote the complete local ring of $\Delta$ at $x_n$; this is a complete DVR with uniformizer $\varphi_q^n(E(u))$, residue field $K_n$, and fraction field $\mathrm{Fr}(\widehat{\fkS}_n)=\widehat{\fkS}_n[1/\varphi_q^n(E(u))]$.

As $E(u)/E(0)\in \fkS_E[1/\pi]$ has constant term $1$, the infinite product
\begin{displaymath}
    \lambda:=\prod_{n\geq 0}\varphi_q^n(E(u)/E(0)).
\end{displaymath}is well-defined as an element in $\O$. By design, $\lambda$ has a simple root at each $x_n$. 
\begin{defn}
A $\varphi$-module (of finite height) over $\fkS_E$ is a finite free $\fkS_E$-module $\fM$ equipped with an isomorphism 
\begin{displaymath}
(\varphi_q^* \fM)[1/E(u)]\simeq \fM[1/E(u)].
\end{displaymath}
We denote by $\mathrm{Vect}^{\varphi}(\fkS)$ the category of $\varphi$-modules over $\fkS$. In analogy with the case $E=\mathbf{Q}_p$, we will often refer to objects in $\mathrm{Vect}^{\varphi}(\fkS)$ as Breuil--Kisin modules over $\fkS_E$.

Similarly, we let $\mathrm{Vect}^{\varphi}(\O)$ denote the category of $\varphi$-modules over $\O$, i.e. finite free $\O$-modules $\M$ equipped with an isomorphism $(\varphi_q^*\M)[1/E(u)]\simeq \M[1/E(u)]$.
\end{defn}
\begin{lem}[Analytic continuation of $\varphi$-modules to the open unit disk]\label{frobenius trick} Base change defines an equivalence of categories
\begin{displaymath}
    \mathrm{Vect}^{\varphi_q}(\O)\simeq \mathrm{Vect}^{\varphi_q}(\fkS_E\langle E(u)/\pi\rangle [1/\pi]).
\end{displaymath}
\end{lem}
\begin{proof}
This is an application of the Frobenius pullback trick. More precisely, as explained in \cite[Rem. 6.6]{prismatic}, by using the contracting property of $\varphi_q$, any object in $\mathrm{Vect}^{\varphi_q}\langle E(u)/\pi\rangle [1/\pi]$ extends uniquely to an object in $\mathrm{Vect}^{\varphi_q}\langle \varphi_q^n(E(u))/\pi\rangle [1/\pi]$ for any $n\geq 1$. It now suffices to show that base change gives an equivalence 
\begin{displaymath}
    \mathrm{Vect}^{\varphi}(\O)\simeq \lim_n\mathrm{Vect}^{\varphi}(\fkS_E\langle \varphi_q^n(E(u))/\pi\rangle [1/\pi]).
\end{displaymath}
This follows formally from the analogous equivalence at the level of vector bundles, and the equality 
\begin{displaymath}
    \O\left[\dfrac{1}{E(u)}\right]=\bigcap_{n\geq 0}\left(\fkS_E\left\langle \dfrac{\varphi_q^n(E(u))}{\pi}\right\rangle \left[\dfrac{1}{\pi}\right]\left[\dfrac{1}{E(u)}\right]\right)\footnote{This follows e.g. from the equality \eqref{functions open disk inverse limit} and Lemma \ref{ring theoretic facts} below.}.
\end{displaymath}
\end{proof}
\subsection{Filtered isocrystals and \texorpdfstring{$\varphi$}{phi}-modules on the open unit disk}\label{construction kisin 06}
In this section, we explain the construction of a natural fully faithful functor 
\begin{displaymath}
\mathrm{MF}^{\varphi_q}(K)\hookrightarrow \mathrm{Vect}^{\varphi_q}(\O).
\end{displaymath}
In fact, inspired by \cite{LafforgueGenestier} and \cite{kim}, we can do slightly better, namely we will construct a commutative diagram
\begin{displaymath}
    \begin{tikzcd}
\mathrm{MF}^{\varphi}(K)\ar[d,hook]\ar[r,hook] & \mathrm{Vect}^{\varphi_q}(\O)\\
\mathrm{HP}^{\varphi_q}(K) \ar[ur,"\simeq",swap], &
    \end{tikzcd}
\end{displaymath}
where $\mathrm{HP}^{\varphi_q}(K)$ denotes the category of $\varphi_q$-modules with Hodge--Pink structure over $K$, whose definition is recalled next.
\begin{defn}
A $\varphi_q$-module with Hodge--Pink structure (or simply a Hodge--Pink isocrystal) over $K$ is a triple $(D,\varphi_q,\Lambda)$ where $(D,\varphi_q)$ is a $\varphi_q$-module over $K_{0,E}$, and $\Lambda$ is a $\widehat{\fkS}_0$-lattice in $D\otimes_{K_{0,E}} \mathrm{Fr}(\widehat{\fkS}_0)$.
\end{defn}
The two categories $\mf(K)$ and $\hp(K)$ are related in the following way. First, given a filtration $\mathrm{Fil}^\bullet D_K$, we get an associated Hodge--Pink structure using the lattice $\Lambda:=\mathrm{Fil}^0(D_K\otimes_K \mathrm{Fr}(\widehat{\fkS}_0))$. Conversely, given a Hodge--Pink lattice $\Lambda$, one gets a filtration on on $D_K$ by first  restricting the $E(u)$-adic filtration $\{E(u)^i \Lambda\}_{i\in\mathbf{Z}}$ on $D\otimes_{K_{0,E}}\mathrm{Fr}(\widehat{\fkS}_0)$ to the tautological lattice $\widehat{D}_0:=D\otimes_{K_{0,E}}\widehat{\fkS}_0$, and then a filtration on $D_K$ by pushing forward along the natural map 
\begin{displaymath}
    \tilde{\theta}: D\otimes_{K_{0,E}}\widehat{\fkS}_0\twoheadrightarrow D\otimes_{K_{0,E}}K=D_K.
\end{displaymath}
\begin{lem}
The resulting functors 
\begin{displaymath}
\begin{tikzcd}
\mf(K)\ar[r,shift left, "P"] & \hp(K)\ar[l,shift left,"F"].
\end{tikzcd}
\end{displaymath}
satisfy $F\circ P\simeq \mathrm{id}$. In particular, $P$ is fully faithful.
\end{lem}
\begin{proof}
Given a filtered isocrystal $(D,\varphi,\mathrm{Fil}^\bullet D_K)$, we need to show that $\tilde{\theta}(\widehat{D}_0\cap E(u)^i\Lambda)=\mathrm{Fil}^i D_K$ for each $i\in\mathbf{Z}$, where $\Lambda:=\mathrm{Fil}^0(D_K\otimes_K \mathrm{Fr}(\widehat{\fkS}_0))$. By shifting, we may assume $i=0$. This desired equality then follows e.g. by choosing a splitting of the given filtration: 
\begin{displaymath}
    D_K =\bigoplus_{j\in\mathbf{Z}}V^j,\quad\mathrm{Fil}^i D_K =\bigoplus_{j\geq i}V^j.
\end{displaymath}
The second statement follows formally from the first and the fact that $F$ is faithful (as it does nothing on the underlying isocrystals). 
\end{proof}
\subsubsection{Constructions}
Thus, combining with Lemma \ref{frobenius trick}, it remains to construct mutually inverse equivalences 
\begin{displaymath}
    \begin{tikzcd}
\hp(K)\ar[r,shift left, "\M"] & \vect(\fkS_E\langle I/\pi\rangle[1/\pi])\ar[l,shift left,"D"].
\end{tikzcd}
\end{displaymath}
For $\M$, we will follow the presentation of \cite[Construction 6.5]{prismatic}, which uses Beauville--Laszlo glueing; one can check that this agrees with Kisin's rather more concrete construction in \cite{Kis06} (when $E=\mathbf{Q}_p$); see Remark \ref{compare with Kisin}. The idea is to define $\M(D)$ as a modification of the constant $\varphi_q$-module $\M':=D\otimes_{K_{0,E}}\fkS_E\langle I/\pi\rangle[1/\pi]$ at the Cartier divisor $\{I=0\}$ (where $I:=(E(u))$) using the given Hodge--Pink lattice $\Lambda$. More precisely, the underlying module of $\M(D)$ is obtained by applying Beauville--Laszlo glueing to the vector bundles
\begin{itemize}
    \item $\M'[1/I]\in \mathrm{Vect}(\fkS_E\langle I/\pi\rangle[1/\pi][1/I])$ 
    \item  $\Lambda\in \mathrm{Vect}(\fkS_I\langle I/\pi\rangle[1/\pi]^{\wedge}_I)$
\end{itemize}
along the obvious isomorphism; here we implicitly identify $\widehat{\fkS}_0$ with $\fkS_E\langle I/\pi\rangle[1/\pi]^{\wedge}_I$ via the natural map (see Lemma \ref{ring theoretic facts} (3) for a more general statement). The $\varphi_q$-structure on $\M(D)$ is then defined as the composition 
\begin{align*}
(\varphi_q^*\M(D))[1/I]\simeq (\varphi_q^*\M')[1/I]\overset{\varphi_{\M'}}{\simeq} \M'[1/I]\simeq \M(D)[1/I];   
\end{align*}
here for the first identification, we use that $\varphi_q(I)$ is a unit in $\fkS_E\langle I/\pi\rangle [1/\pi]$.  

Next, we define the functor $D$. Given $\M\in \vect(\fkS_E\langle I/\pi\rangle [1/\pi])$, set $D(\M):=\M\otimes_{\fkS_E\langle I/\pi\rangle [1/\pi]}K_{0,E}$, equipped with the natural (diagonal) $\varphi_q$-action; here $\fkS_E\langle I/\pi\rangle\twoheadrightarrow K_{0,E}$ is the natural ($\varphi_q$-equivariant) map $u\mapsto 0$. This gives the isocrystal structure on $D(\M)$; it remains to define the Hodge--Pink lattice. First, the standard Frobenius trick shows that there is a unique $\varphi$-equivariant map 
\begin{displaymath}
    \xi: D(\M)\otimes_{K_{0,E}}\fkS_E\langle I/\pi\rangle [1/\pi][1/I]\to \M[1/I]
\end{displaymath}
lifting the identity modulo $u$. See \cite[Lemma 1.2.6]{Kis06} or \cite[Lemma 3.5]{LafforgueGenestier}. In particular, $\xi$ realizes $\M$ as a modification of $D(\M)\otimes_{K_{0,E}}\fkS_E\langle I/\pi\rangle [1/\pi]$ at the divisor $\{I=0\}$, and hence $\M^{\wedge}_{I}$ gives rise to the desired lattice inside $\fM^{\wedge}_I[1/I]\simeq D\otimes_{K_{0,E}}\mathrm{Fr}(\widehat{\fkS}_0)$. While this already finishes the construction of the functor $\M\mapsto D(\M)$, we note that the associated filtration on $D(\M)_K$ can be maded slightly more explicit as follows. Indeed, as $\varphi(I)$ is a unit in $\fkS_E\langle I/\pi\rangle [1/\pi]$ (as was already mentioned), $\varphi_q^*\xi$ is an isomorphism, fitting in following commutative square
\begin{displaymath}
\begin{tikzcd}
    \varphi_q^* D(\M)\otimes_{K_{0,E}} \fkS_E\langle I/\pi\rangle [1/\pi]\ar[d,hook,"\varphi_{D(\M)}",swap]\ar[r,"\varphi_q^*\xi", "\simeq"'] & \varphi_q^*\M \ar[d,hook,"\varphi_{\M}"]\\
    D(\M)\otimes_{K_{0,E}} \fkS_E\langle I/\pi\rangle [1/\pi][1/I] \ar[r,"{\xi}"', "\simeq"] & \M[1/I].
\end{tikzcd}
\end{displaymath}
In particular, we see that $D(\M)\otimes_{K_{0,E}} \widehat{\fkS}_0\simeq \varphi_q^*D(\M)\otimes_{K_{0,E}} \widehat{\fkS}_0\simeq (\varphi_q^*\M)^{\wedge}_I$. The filtration on $D(\M)_K=D(\M)\otimes_{K_{0,E}} K$ is simply the image of the filtration $\mathrm{Fil}^\bullet (\varphi_q^*\M):=\varphi_{\M}^{-1}(I^\bullet\M)$ on $\varphi_q^*\M$. This agrees with the filtration constructed in \cite[\textsection 1.2.7]{Kis06} (when $E=\mathbf{Q}_p$).
\begin{theorem}\label{mutual inverse phi-module hodgepink}
The functors 
\begin{displaymath}
    \begin{tikzcd}
\hp(K)\ar[r,shift left, "\M(\cdot)"] & \vect(\fkS_E\langle I/\pi\rangle[1/\pi])\ar[l,shift left,"D(\cdot)"].
\end{tikzcd}
\end{displaymath}
are mutually inverse equivalences of categories. 
\end{theorem}
\begin{proof}
This follows readily from the construction of the functors. Here we will only explain the proof that $\M\circ D\simeq \mathrm{id}$; that $D\circ \M\simeq \mathrm{id}$ can be proved similarly. Fix $\M\in\mathrm{Vect}^{\varphi}(\fkS_E\langle I/\pi\rangle [1/\pi])$. As $\M^{\wedge}_I=\Lambda_{D(\M)}$ by construction of $D(\M)$, we see that $\M(D(\M))\simeq \M$ on underlying modules. To compare the $\varphi$-structures, recall that by construction of $\M(\cdot)$, $\varphi_{\M(D(\M))}$ is the unique map $\varphi_q^*\M(D(\M))\to \M(D(\M))[1/I]$ making the diagram  
\begin{displaymath}
\begin{tikzcd}
    \varphi_q^* D(\M)\otimes_{K_{0,E}} \fkS_E\langle I/\pi\rangle [1/\pi]\ar[d,hook,"\varphi_{D(\M)}",swap] \ar[r,"\simeq"] & \varphi_q^*\M(D(\M)) \ar[d,hook,dashed,"\varphi_{\M(D(\M))}"]\\
    D(\M)\otimes_{K_{0,E}} \fkS_E\langle I/\pi\rangle [1/\pi][1/I] \ar[r,"{\xi}"',"\simeq"] & \M(D(\M))[1/I].
\end{tikzcd}
\end{displaymath}
commute, but $\varphi_{\M}$ is also such a map (by $\varphi$-equivariance of $\xi$, as we explained above), so they coincide, as wanted.
\end{proof}
\begin{remark}[Comparison with \cite{Kis06}]\label{compare with Kisin}
In this remark, we briefly check that the compositions $\M\circ P$ and $F\circ D$ agree with the ones from \cite{Kis06} (in case $E=\mathbf{Q}_p$), up to composing with the natural equivalence from Lemma \ref{frobenius trick}. Denote the latters by $\M'$ and $D'$. That $D'\simeq F\circ D$ is already explained in the paragraph above Theorem \ref{mutual inverse phi-module hodgepink}. We now show that $\M'\simeq \M\circ P$. Fix $D:=(D,\varphi,\mathrm{Fil}^\bullet D_K)\in \mf(K)$. Note firstly that the ring $\widehat{\fkS}_n$ from \cite[\textsection 1.1.1]{Kis06} is not our $\widehat{\fkS}_{n}$, rather it is $\varphi_W^{-n}(\widehat{\fkS}_n)$, where $\varphi_W: \O\to \O$ is the automorphism given by Frobenius on $W(k)$ (and $u\mapsto u$). Using this observation, it is easy to rewrite the definition of $\M'(D)$ in \cite[\textsection 1.2]{Kis06} as 
\begin{displaymath}\label{kisin description}
    \M'(D):=\{x\in D\otimes_{K_0}\O[1/\lambda]\;|\;\iota_n(x)\in \mathrm{Fil}^0(D_K\otimes_K \mathrm{Fr}(\widehat{\fkS}_n))\;\text{for all $n\geq 0$}\},
\end{displaymath}
where $\iota_n$ is the natural map $D\otimes_{K_0}\O[1/\lambda]\to D\otimes_{K_0}\mathrm{Fr}(\widehat{\fkS}_n)$ (which indeed makes sense as $\lambda$ has a simple root at each $x_n$). Then $\M'(D)\subseteq D\otimes_{K_0}[1/\lambda]$ is a finite free sub-$\O$-module with $\M'(D)[1/\lambda]=D\otimes_{K_0}\O[1/\lambda]$; moreover the isomorphism $\varphi_D: \varphi^*D\otimes_{K_0}\O[1/\lambda]\simeq D\otimes_{K_0}\O[1/\lambda]$ restricts to an isomorphism $(\varphi^*\M'(D))[1/E(u)]\simeq \M'(D)[1/E(u)]$, making $\M'(D)$ an object in $\mathrm{Vect}^{\varphi}(\O)$. See \cite[Lemma 1.2.2]{Kis06}. In particular, by base change along the natural map $\O[1/\lambda]\to \fkS_E\langle I/p\rangle [1/p][1/I]$ (which makes sense as $\varphi^n(I)$ is invertible in $\fkS_E\langle I/p\rangle [1/p]$ for all $n\geq 1$), we obtain 
\begin{displaymath}
    (\M'(D)\otimes_{\O} \fkS_E\langle I/p\rangle [1/p])[1/I]\simeq D\otimes_{K_0}\fkS_E\langle I/p\rangle [1/p][1/I]. 
\end{displaymath}
Moreover, the description \ref{kisin description} also shows that $\M'(D)^{\wedge}_I$ identifies with $\mathrm{Fil}^0(D_K\otimes_K \mathrm{Fr}(\widehat{\fkS}_0))=:\Lambda_{F(D)}$ (e.g. by applying $\varphi_W^n$ to \cite[Lemma 1.2.1 (2)]{Kis06}). This shows that 
\begin{displaymath}
    \M'(D)\otimes_{\O}\fkS_E\langle I/p\rangle [1/p]\simeq \M(F(D))\;\text{in $\mathrm{Vect}^{\varphi}(\O)$},
\end{displaymath}
as claimed. 
\end{remark}
\subsection{Slopes and weak admissibility}
Similarly to \cite{Kis06} and \cite{kim}, in this subsection we relate, following Berger's oberservation \cite{berger08}, weakly admissbility of Hodge--Pink isocrystals, and the ``pure of slope $0$'' condition for $\varphi$-modules on the open unit disk. As many of the arguments are identical to those in \cite{Kis06}, we often sketch only the proofs.

Recall firstly the notion of weak admissibility for Hodge--Pink isocrystals. Let $D:=(D,\varphi,\Lambda)\in \hp(K)$. The Newton number $t_N$ of $D$ is defined exactly as before (i.e. using the underlying isocrystal). For defining the Hodge number, again by passing to the determinant, we may assume $D$ is 1-dimensional; in this case we set $t_H(D):=h$, where $h$ is the unique integer such that $\Lambda=(E(u))^{-h}(D\otimes_{K_{0,E}}\widehat{\fkS}_0)$ (so $t_H(D)$ only depends on the Hodge--Pink structure of $D$).
\begin{defn}
A Hodge--Pink $\varphi_q$-module $D=(D,\varphi_q,\Lambda)$ is called weakly admissibile if $t_H(D)=t_N(D)$ and $t_H(D')\leq t_N(D')$ for all subojects $D'\subseteq D$ in $\hp(K)$\footnote{The Hodge--Pink lattice on a subobject $D'\subseteq D$ is by definition given by $\Lambda_{D'}:=\Lambda_D\cap (D'\otimes_{K_{0,E}}\mathrm{Fr}(\widehat{\fkS}_0))$ (just as a subobject in $\mf(K)$ is endowed with the subspace fitration). However, the notion of weak admissibility does not change if we weaken this into the weaker condition that $\Lambda_{D'}\subseteq \Lambda_D$.}.
\end{defn}
\begin{lem}\label{weak admissibility preserve}
The (fully faithful) functor
\begin{displaymath}
\begin{tikzcd}
\mf(K)\ar[r,hook, "P"] & \hp(K)
\end{tikzcd}
\end{displaymath}
preserves weak admissibility. More precisely, an object $D\in \mf(K)$ is weakly admissible if and only if its image $P(D)$ is. 
\end{lem}
\begin{proof}
First, $F$ and $P$ both preserve the Newton numbers $t_N$ as they do nothing on the underlying isocrystals. They also preserve the Hodge numbers: for this we may reduce to the rank $1$ case, where the result follows by a direct computation. It follows that an object $D\in \mf(K)$ (resp. 
$D'\in \hp(K)$) is weakly admissible if $P(D)$ (resp. $F(D')$) is so. Moreover, as $F\circ P\simeq \mathrm{id}$, it then follows that $P(D)$ is weakly admissible whenever $D$ is so.  
\end{proof}
\begin{remark}
It is however \textit{not} true that $F$ preserves weakly admissible objects. The following example is taken from \cite{LafforgueGenestier}. Consider the object $D=(D,\varphi,\Lambda)\in \hp(K)$ with $D=K_0 e_1\oplus K_0 e_2, \varphi(e_i)=e_i$, and Hodge--Pink lattice
\begin{displaymath}
    \Lambda:=E(u)^{-1}\widehat{\fkS}_0e_1\oplus \widehat{\fkS}_0 e_2.
\end{displaymath}
One can check directly that $(D,\varphi,\Lambda)$ is weakly admissible. On the other hand, the associated filtered isocrystal is given by
\begin{displaymath}
    \mathrm{Fil}^0 D_K=D_K,\quad\mathrm{Fil}^1 D_K=Ke_1, \quad\mathrm{Fil}^2 D_K=0,
\end{displaymath}
which is not weakly admissible as the submodule $D':=K_0 e_1$ has $t_N(D',\varphi)=0$ but $t_H(\mathrm{Fil}^\bullet D'_K)=1$. In particular, we see that $F$ and $P$ are not equivalences of categories (though they are on rank $1$ objects). 
\end{remark}
\subsubsection{Kedlaya's slope filtration}
Let 
\begin{displaymath}
    \R:=\varinjlim_{r\mapsto 1^{-}}\O_{(r,1)}
\end{displaymath}
be the Robba ring over $K_{0,E}$; here $\O_{(r,1)}$ denotes the ring of rigid analytic functions on the open annulus $\{r<|u|<1\}$. By work of Lazard, $\R$ is a B\' ezout domain containing $\O$ as a subring. Again, there is a natural Frobenius map $\varphi_q: \R\to \R$ extending $\varphi_q$ on $\O$. Inside $\R$, there is the subring $\R^b$ formed of functions which are bounded. This is a Henselian discrete valued field with uniformizer $\pi$, and ring of integers
\begin{displaymath}
    \R^{int}=\{\sum_{n\in\mathbf{Z}} a_n u^n\in \R\;|\;u_n\in W_{\O_E}(k)\;\text{for all $n\in\mathbf{Z}$}\},
\end{displaymath}
cf. \cite[\textsection 7.2]{DurhamFargues}. In particular, the $p$-adic completion of $R^{int}$ identifies with $\fkS_E[1/E(u)]^{\wedge}_p=:\O_{\E}$. Clearly, both $\R^b$ and $\R^{int}$ are $\varphi_q$-stable inside $\R$.

We denote by $\mathrm{Vect}^{\varphi_q}(\R)$ the category of $\varphi$-modules over $\R$, i.e. finite free $\R$-modules $\M$ equipped with an isomorphism $\varphi_q^*\M\simeq \M$. A $\varphi$-module $\M$ is called \' etale or pure of slope $0$ if it contains a $\varphi_q$-stable $\R^{int}$-lattice $\N$ for which the map $\varphi_q^*\N\to \N$ is an isomorphism. By twisting suitably with a rank $1$ module, one can then define the subcategory $\mathrm{Vect}^{\varphi_q,s}(\R)$ of objects pure of slope $s$ for any $s\in\mathbf{Q}$; see \cite[Definition 1.6.1]{kedlaya3relative}. Similarly, we denote by $\mathrm{Vect}^{\varphi_q,s}(\R^b)$ the subcategory of $\varphi_q$-modules over $\R^b$ which are pure of slope $s$ in the sense of Dieudonn\' e--Manin theory (recall that $\R^b$ is a discrete valued field). Finally, a $\varphi_q$-module $\M$ over $\O$ is called pure of slope $0$ if $\M\otimes_{\O}\R$ is.
\begin{thm}\label{kedlaya theorem}
(1) Base change defines an equivalence of categories 
\begin{displaymath}
    \mathrm{Vect}^{\varphi_q,s}(\R^b)\simeq \mathrm{Vect}^{\varphi_q,s}(\R).
\end{displaymath}
(2) For any $\M\in \mathrm{Vect}^{\varphi_q}(\R)$, there exists a unique filtration 
\begin{displaymath}
    0=\M_0\subset \M_1\subset\ldots\subset \M_{r}=\M,
\end{displaymath}
in $\mathrm{Vect}^{\varphi_q}(\R)$, called the slope filtration, such that the quotient $\M_i/\M_{i-1}$ is (finite free and) pure of slope $s_i\in\mathbf{Q}$ and $s_1<s_2<\ldots<s_r$.
\end{thm}
\begin{proof}
    See \cite[Theorem 1.6.5]{kedlaya3relative} and \cite[Theorem 1.7.1]{kedlaya3relative}.
\end{proof}
\begin{prop}\label{slope filtration descend}
Let $\M\in \mathrm{Vect}^{\varphi_q}(\O)$. Then the slope filtration on $\M\otimes_{\O}\R$ descends uniquely to a filtration on $\M$ by (saturated\footnote{A submodule $\N\subseteq \M$ is called saturated if it is a direct summand of $\M$.}) subobjects in $\vect(\O)$.
\end{prop}
\begin{proof}
The arguments in \cite[\textsection 4.2]{kim} carry over to our setting.
\end{proof}
\begin{remark}
Note that, unlike the proof of \cite[Proposition 1.3.7]{Kis06}, which relies crucially on a monodromy operator, the proof of \cite{kim} is intrinsic in the world of $\varphi$-modules.
\end{remark}
We can now translate the weak admissibility condition for Hodge--Pink isocrystals across the equivalence of categories in Theorem \ref{mutual inverse phi-module hodgepink}.
\begin{thm}\label{weak admissibility and slope 0}
    Let $D:=(D,\varphi_q,\Lambda)\in \hp(K)$. Then $D$ is weakly admissible if and only if $\M(D)$ is pure of slope $0$.
\end{thm}
Here in the statement we implicitly identify $\M(D)$ with the corresponding $\varphi_q$-module over $\O$ (Lemma \ref{frobenius trick}).
\begin{proof}
Assume first that $D$ has rank $1$. In this case
\begin{align*}
     \M(D) =D\otimes_{K_{0,E}}\lambda^{-t_H(D)}\O;
\end{align*}
see e.g. Remark \ref{compare with Kisin}. Pick a basis $e\in D$ and write $\varphi_D(e)=\alpha e$ for some $\alpha\in K_0$. Then  
\begin{displaymath}
    \varphi_q(e\otimes \lambda^{-t_H(D)})=(E(u)/E(0))^{t_H(D)}\alpha(e\otimes \lambda^{-t_H(D)});
\end{displaymath}
as $E(u)$ is a unit in $\R^{int}$, we see that $\M(D)$ has slope $t_N(D)-t_H(D)$. This proves the theorem for rank 1 objects. The general case then follows by the same argument as in \cite[Theorem 1.3.8]{Kis06}, using the equivalence in Theorem \ref{mutual inverse phi-module hodgepink} and Proposition \ref{slope filtration descend} in place of \cite[Proposition 1.3.7]{Kis06}.
\end{proof}
It will be convenient to state the following lemma separately. 
\begin{lem}\label{beauville laszlo glueing and extending vbs}
Base change defines an equivalence of categories
    \begin{displaymath}
        \mathrm{Vect}(\fkS_E)\simeq \mathrm{Vect}(\fkS_E[1/p])\times_{\mathrm{Vect}(\E)}\mathrm{Vect}(\O_{\E}),
    \end{displaymath}
    where $\O_{\E}:=\fkS_E[1/E(u)]^{\wedge}_p$ and $\E:=\O_{\E}[1/p]$. Moreover, this induces an equivalence\footnote{Using a result of Kedlaya \cite[Lemma 4.6]{BMS1}, one sees by the same argument that the lemma also holds for the perfectoid variant $A_{\inf,E}:=W_{\O_E}(\O_C^\flat)$ of $\fkS_E$.}  
    \begin{displaymath}
        \vect(\fkS_E)\simeq \vect(\fkS_E[1/p])\times_{\vect(\E)}\vect(\O_{\E}).
    \end{displaymath}
\end{lem}
\begin{proof}
It suffices to show the first assertion, which follows from Beauville--Laszlo glueing, and the well-known fact that restricting gives an equivalence 
\begin{displaymath}
\mathrm{Vect}(\Sp(\fkS_E))\simeq \mathrm{Vect}(\Sp(\fkS_E)-\{\m\})
    \end{displaymath}
    (note that $\fkS_E$ is a $2$-dimensional regular local ring). 
\end{proof}
\begin{prop}\label{slope 0 and BK modules}
Base change defines an equivalence
    \begin{displaymath}
        \mathrm{Vect}^{\varphi_q}(\fkS_E)[1/p]\simeq \mathrm{Vect}^{\varphi_q,0}(\O),
    \end{displaymath}
    where the source denotes the isogeny category of $\mathrm{Vect}^{\varphi_q}(\fkS_E)$.
\end{prop}
\begin{proof}
See the proof of \cite[Lemma 1.3.13]{Kis06}. 
\end{proof}
\begin{cor}\label{fully faithful into Breuil Kisin}
    There is a natural fully faithful functor
    \begin{displaymath}
        \begin{tikzcd}
        \mathrm{MF}^{\varphi_q,w.a}(K)\ar[r,hook,"{\fM}(\cdot)"] & \mathrm{Vect}^{\varphi_q}(\fkS_E)[1/p].
        \end{tikzcd}
    \end{displaymath}
\end{cor}
\begin{proof}
    This follows by combining Lemma \ref{weak admissibility preserve}, Theorem \ref{weak admissibility and slope 0}, and Proposition \ref{slope 0 and BK modules}.
\end{proof}
\begin{prop}\label{fully faithful etale realization Kisin}
    The base change functor 
    \begin{displaymath}
        \vect(\fkS_E)\to \vect(\fkS_E[1/E(u)]^{\wedge}_p)
    \end{displaymath}
    is fully faithful. 
\end{prop}
\begin{proof}
With the previous results in place, the proof is similar to that of \cite[Proposition 2.1.12]{Kis06}. Namely, as in \textit{loc. cit.}, it suffices to show that if $h: \fM_1\to \fM_2$ is a map in $\vect(\fkS_E)$ such that $h[1/E(u)]^{\wedge}_p$ is an isomorphism, then $h$ is an isomorphism. We may assume $\fM_1$ and $\fM_2$ are free of rank $1$. By Lemma \ref{beauville laszlo glueing and extending vbs}, it suffices to show that $h[1/p]$ is an isomorphism. This follows by using the embedding $\vect(\fkS_E)[1/p]\hookrightarrow \mathrm{HP}^{\varphi_q, w.a.}(K)$, and the fact that a map of \textit{weakly admissible} objects in $\hp(K)$ which is an isomrphism on underlying modules (e.g. a nonzero map between rank 1 objects) is necessarily an isomorphism.
\end{proof}
\section{Prismatic \texorpdfstring{$F$}{F}-crystals and \texorpdfstring{$E$}{E}-crystalline Galois representations}\label{prismatic story}
\subsection{Preliminaries on \texorpdfstring{$\O_E$}{OE}-prisms}
We keep the notation as before. In particular, we fix throughout a uniformizer $\pi$ of $E$, and an embedding $\tau_0: E\hookrightarrow K\subseteq C$.

\begin{defn}[{{\cite[Definition 3.1]{marks2023prismatic}}}]
Let $X$ be a $\pi$-adic formal scheme over $\Spf \O_E$. The (absolute) prismatic site $(X)_{\Prism,\O_E}$ is by definition the site with whose objects are bounded $\O_E$-prisms $(A,I)$ with a map $\Sp(A/I)\to X$ of $\pi$-adic $\O_E$-formal schemes, with coverings given by maps of prisms whose underlying ring map is $(\pi,I)$-adically completely faithfully flat.
\end{defn}
For the precise definition of $\O_E$-prisms, we refer the reader to \cite{marks2023prismatic} (see also \cite{ito2023prismatic}). As we will work entirely with $\O_E$-prisms in what follows, we will typically drop $\O_E$ from the notation.

We mention here some examples of $\O_E$-prisms that are most relevant for our purpose. 
\begin{example}\label{example OE prisms}
    (1) (Breuil--Kisin prisms) Choose a uniformizer $\pi_K\in K$. As $\fkS_E:=W_{\O_E}(k)[[u]]$, endowed with the $\delta_E$-structure given by $\delta_E(u)=0$ (or equivalently, $\varphi_q(u)=u^q$). Let $E(u)\in W_{\O_E}(k)[u]$ be the Eisenstein polynomial of $\pi_K$ over $K_{0,E}$; here $K_{0,E}:=K_0\otimes_{E_0} E$ denotes the maximal unramified extension of $\tau_0: E\subseteq K$. As the map $\widetilde{\theta}: \fkS_E\twoheadrightarrow \O_K, u\mapsto \pi_K$ is surjective with kernel $I=(E(u))$, the pair $(\fkS_E,I)$ gives an object in $(\O_K)_{\Prism,\O_E}$, which we will refer to as the Breuil--Kisin prism associated to the chosen uniformizer $\pi_K$.

    (2) (The $A_{\inf,E}$-prism) Recall that $C$ denotes a fixed completed algebraic closure of $K$. Set $A_{\inf,E}:=W_{\O_E}(\O_C^\flat)$, equipped with the natural Frobenius $\varphi_q$. As usual, the Fontaine's theta map $\theta_E: A_{\inf,E}\to \O_C$ is surjective with kernel generated by a nonzero-divisor $\xi$. The twisted map $\widetilde{\theta}_E:=\theta_E\circ \varphi_q^{-1}$ is thus also surjective with kernel $(\widetilde{\xi})$ where $\widetilde{\xi}:=\varphi_q(\xi)$. In particular, the pair $(A_{\inf,E},(\widetilde{\xi}))$ defines an object in $(\O_K)_{\Prism,\O_E}$, and is the $\O_E$-prism corresponding to the perfectoid $\O_E$-algebra $\O_C$, i.e. $A_{\inf,E}=\Prism_{\O_C}$. For later use, we give here an explicit choice of $\xi$. Let $v=(v_0,v_1,\ldots)\in T\G$ be a generator of the Tate module of the Lubin--Tate group $\G$ of $E$ associated to some Frobenius polynomial $Q\in \O_E[T]$ for $\pi$. As $v_{n+1}^q\equiv v_n\bmod{\pi}$ for all $n$, we obtain an element 
    \begin{displaymath}
        v:=(v_0\bmod \pi,v_1\bmod \pi,\ldots)\in \varprojlim_{x\mapsto x^q}\O_{C}/\pi=\O_{C^\flat}.
    \end{displaymath}
    (The last identification is given by $(a_n)_n\mapsto (a_0,a_1^{p^{f-1}},\ldots,a_1^p,a_1,a_2^{p^{f-1}},\ldots)$.) Following \cite[Proposition 1.2.7]{FFCurves}, for a perfect $\mathbf{F}_q$-algebra $A$, we denote by $[\cdot]_Q$ (or $[\cdot]_{\G}$) the unique map $A\to W_{\O_E}(A)$ satisfying $[x]_Q\equiv x\bmod{\pi}$ and $Q([x]_Q)=\varphi_q([x]_Q)$. When $E=\mathbf{Q}_p,\pi=p$ and $Q(T)=(1+T)^p-1$, $[x]_Q$ is nothing but $[x+1]-1$ and hence $[v]_Q=[\epsilon]-1$ is the usual element $\mu$; accordingly, we will also write $\mu:=[v]_Q$ here. One can then check that $\xi:=\mu/\varphi_q^{-1}(\mu)$ is a generator of $\ker(\theta_E)$.  

    Exactly as in \cite[Example 2.6]{prismatic}, one can show that both examples above give covers of the final object in the topos $\mathrm{Shv}((\O_K)_{\Prism})$. Moreover, there is a map $\fkS_E\to \Prism_{\O_C}$ in $(\O_K)_{\Prism,\O_E}$, defined by $u\mapsto [\pi_K^\flat]$, where $\pi_K^\flat:=(\pi_K,\pi_K^{1/q},\ldots)\in \O_C^\flat$ is a compatible system of $q$-power roots of the fixed uniformizer $\pi_K$. 

    (3) (The $A_{\mathrm{cris},E}$-prism) Recall from Remark \ref{Bcris E} that $A_{\mathrm{cris},E}$ denotes the $\pi$-completed $\O_E$-PD envelope of $A_{\inf,E}$ with respect to the kernel of $\theta_E$, $B_{\mathrm{cris},E}^+:=A_{\mathrm{cris},E}[1/\pi]$, and $B_{\mathrm{cris},E}=B_{\mathrm{cris},E}^+[1/t_E]$. By \cite[Proposition 2.6.5]{ito2023prismatic}, the pair $(A_{\mathrm{cris},E},(\pi))$ identifies with the prismatic envelope $\Prism_{\O_C}\{\widetilde{\xi}/\pi\}$, hence also defines an object in $(\O_K)_{\Prism,\O_E}$, which we denote by $\Prism_{\O_C/\pi}$ in analogy with the usual case\footnote{While it is reasonable to define a notion of qrsp rings and their associated prisms in the ``$\O_E$-context'', for our purpose it suffices to see this as a purely suggestive notation.}; of course there is a natural map $\Prism_{\O_C}\to \Prism_{\O_C/\pi}$.
\end{example}
\begin{remark}[$W_{\O_E}(k)$-algebra structure on objects in $(\O_K)_{\Prism,\O_E}$]\label{algebra structure Wk}
    Fix an object $(A,I)\in (\O_K)_{\Prism,\O_E}$ with structure map $\O_K\to A/I$. By standard deformation theory, the composition $W_{\O_E}(k)\to \O_K\to A/I$ lifts uniquely to an $\O_E$-algebra map $W_{\O_E}(k)\to A$. In what follows, we will always regard an object in $(\O_K)_{\Prism,\O_E}$ as an $W_{\O_E}(k)$-algebra via this map. (Note that for the prism $\Prism_{\O_C}=W_{\O_E}(\O_C^\flat)$, this is not the ``natural'' structure (induced by the canonical section $k\to \O_C^\flat$) but its $\varphi_q$-twist: the point is that we are taking into account not only the underlying ring, but also the invertible ideal defining the prism structure.) By uniqueness, morphisms in $(\O_K)_{\Prism,\O_E}$ automatically respect this algebra structure.
\end{remark}
\begin{notation}[Some period sheaves on the prismatic site] 
We consider the following period sheaves on $X_{\Prism}$: 
\begin{itemize}
    \item The prismatic structure sheaf $\O_{\Prism}: (A,I)\mapsto A$; this comes equipped with an ideal sheaf $I_{\Prism}: (A,I)\mapsto I$ and a ``Frobenius'' $\varphi_q: \O_{\Prism}\to \O_{\Prism}$.

    \item The \' etale structure sheaf $\O_{\Prism}[1/{\I_{\Prism}}]^{\wedge}_{\pi}$:
    \begin{displaymath}
        (A,I)\mapsto A[1/I]^{\wedge}_{\pi}.
    \end{displaymath}
    \item The rational localization $\O_{\Prism}\langle \I_{\Prism}/\pi\rangle$: 
    \begin{displaymath}
        (A,I) \mapsto A[ I/\pi]^{\wedge}_{\pi}.
    \end{displaymath}
    \item The de Rham period sheaves: 
    \begin{displaymath}
    \mathbb{B}_{\mathrm{dR}}^+:=(\O_{\Prism}[1/\pi])^{\wedge}_{\I_{\Prism}}\quad\text{and}\quad \mathbb{B}_{\mathrm{dR}}:=\mathbb{B}_{\mathrm{dR}}^+[1/\I_{\Prism}].
\end{displaymath}
\end{itemize}
It is easy to see that the Frobenius on $\O_{\Prism}$ extends naturally to the sheaves $\O_{\Prism}[1/{\I_{\Prism}}]^{\wedge}_{\pi}$ and $\O_{\Prism}\langle \I_{\Prism}/\pi\rangle$. (Note again that in case $X=\Spf(\O_K)$, the value $\mathbb{B}_{\mathrm{dR}}^+(\Prism_{\O_C})=:B_{\mathrm{dR}}^+$ is a $\varphi_q$-twist of the ring denoted by the same notation in the appendix.)
\end{notation}
\begin{defn}
A prismatic $F$-crystal on $X$ is a pair $(\E,\varphi_{\E})$ where $\E$ is a crystal of vector bundles on the ringed site $(X_{\Prism},\O_{\Prism})$, and $\varphi_{\E}$ is an isomorphism $(\varphi_q^*\E)[1/\I_{\Prism}]\cong \E[1/\I_{\Prism}]$. The resulting category is denoted by $\mathrm{Vect}^{\varphi_q}(X_{\Prism},\O_{\Prism})$.

More generally, for a sheaf $\O'$ of $\O_{\Prism}$-algebras equipped with a compatible Frobenius, we define similarly the category $\vect(X_{\Prism},\O')$ of $F$-crystals over $\O'$ on $X$. Similarly, if $(A,I)$ is an $\O_E$-prism, we define in the same way the category $\vect(A,I)$ (or simply $\vect(A)$) of $F$-crystals (or Breuil--Kisin modules) over $A$.  
\end{defn} 
\begin{remark}
As descent for vector bundles is effective for the flat topology, to give a crystal of vector bundles on $(X_{{\Prism}},\O_{\Prism})$ is to give for each object $(A,I)$ in $X_{\Prism}$, a finite projective $A$-module $M_A$, and for each map $(A,I)\to (B,J)$ in $X_{\Prism}$, an isomorphism $M_A\otimes_AB\xrightarrow{\sim}M_B$ compatible with compositions. In other words,
\begin{displaymath}
    \mathrm{Vect}^{(\varphi_q)}(X_{\Prism},\O_{\Prism})\simeq \lim_{(A,I)\in X_{\Prism}}\mathrm{Vect}^{(\varphi_q)}(A,I).
\end{displaymath}
A similar result holds for the sheaves $\O_{\Prism}[1/\I_{\Prism}]^{\wedge}_\pi$ and $\O_{\Prism}\langle \I_{\Prism}/\pi\rangle [1/\pi]$ (for the first see \cite[Propposition 2.7]{prismatic}; for the second see the proof of \cite[Corollary 7.17]{prismatic}). 
\end{remark}
\subsection{Formulation of the main theorem}
We now restrict ourselves to the case $X=\Spf(\O_K)$, viewed as an $\O_E$-formal scheme using the fixed embedding $\tau_0: E\hookrightarrow K$.

Recall from \cite{marks2023prismatic} that there is a natural equivalence
\begin{align*}
     T: \vect((\O_K)_{\Prism},\O_{\Prism}[1/\I_{\Prism}]^{\wedge}_{\pi}) &\simeq \mathrm{Rep}_{\O_E} (G_K)\\
     \E &\mapsto (\E(\Prism_{\O_C}))^{\varphi_q=1}.
\end{align*}
In particular, by extending scalars along $\O_{\Prism}\to \O_{\Prism}[1/\I_{\Prism}]^{\wedge}_{\pi}$, we obtain a functor
\begin{displaymath}
    T: \mathrm{Vect}^{\varphi_q}(X_{\Prism},\O_{\Prism}) \to \mathrm{Rep}_{\O_E}(G_K),
\end{displaymath}
which we again refer to as the \' etale realization functor. We can now state our main result. 
\begin{thm}\label{main thm body}
The \' etale realization functor gives rise to an equivalence of categories 
\begin{displaymath}
T: \mathrm{Vect}^{\varphi_q}(X_{\Prism},\O_{\Prism})\to \mathrm{Rep}^{\mathrm{crys}}_{\O_E}(G_K)
\end{displaymath}
where the target denotes the category of finite free $\O_E$-modules $T$ equipped with a continuous linear $G_K$-action such that $T[1/\pi]$ is $E$-crystalline.
\end{thm}
\begin{proof}
We first show that the \' etale realization $T(\E)$ of a prismatic $F$-crystal on $\O_K$ is indeed an object in the target; full faithfulness and essential surjectivity will be dealt separately below. We will follow the proof of \cite[Proposition 5.3]{prismatic}. First by the crystal structure of $\E$, we have a natural isomorphism 
\begin{displaymath}
\E(\Prism_{\O_C})\otimes_{\Prism_{\O_C}}\Prism_{\O_C/\pi}\xrightarrow{} \E(\Prism_{\O_C/\pi}).
\end{displaymath}
Also, by Lemma \ref{BKF} below, we have a natural identification 
\begin{displaymath}
    T(\E)\otimes_{\O_E}\Prism_{\O_C}[1/\mu]=\E(\Prism_{\O_C})[1/\mu]
\end{displaymath}
Pick $n\gg 0$ so that the map $\varphi_q^n: \O_K/\pi\to \O_K/\pi$ factors through the natural reduction map $\O_K/\pi\twoheadrightarrow k$ (where $k$ denotes the residue field of $K$). In particular, we see that the natural map $k\to \O_C/\pi$ is $\O_K$-linear when the target is now regarded as an $\O_K$-algebra via the map $\varphi_q^n: \O_K/\pi\to \O_C/\pi$. This lifts to a map $W_{\O_E}(k)\to \Prism_{\O_C/\pi}'$ in $(\O_K)_{\Prism}$, where $\Prism_{\O_C}'$ denotes the object in $(\O_K)_{\Prism}$ with underlying prism $\Prism_{\O_C/\pi}$ but endowed with the map $\O_K\xrightarrow{\varphi_q^n} \O_C/\pi \to \Prism_{\O_C/\pi}/(\pi)$ (thus the Frobenius on $\Prism_{\O_C/\pi}$ defines a map $\varphi_q^n: \Prism_{\O_C/\pi}\to \Prism_{\O_{C}/\pi}'$ in $(\O_K)_{\Prism}$). Thus, by using the crystal and Frobenius structures of $\E$, we obtain a natural isomorphism 
\begin{align*}
    \E(W_{\O_E}(k))\otimes_{W_{\O_E}(k)}\Prism_{\O_C/\pi}[1/\pi] & \simeq \E(\Prism_{\O_{C}/\pi}')[1/\pi]\\
    &\simeq (\varphi_q)^*\E(\Prism_{\O_C/\pi})[1/\pi]\\
    &\simeq \E(\Prism_{\O_C/\pi})[1/\pi].
\end{align*}
Note also that as $W_{\O_E}(k)$ is fixed by the $G_K$-action on $\Prism_{\O_C/\pi}'$ under the map $W_{\O_E}(k)\to \Prism_{\O_C/\pi}'$, the crystal property of $\E$ again implies that $G_K$ acts trivially on $\E(W_{\O_E}(k))$.
Putting things together, we obtain a $G_K$-equivariant isomorphism
\begin{displaymath}
    T(\E)\otimes_{\O_E}B_{\mathrm{cris},E}\simeq \E(W_{\O_E}(k))\otimes_{W_{\O_E}(k)}B_{\mathrm{cris},E},
\end{displaymath}
with $G_K$ acting trivially on $\E(W_{\O_E}(k))$. By Theorem \ref{characterization of E-crystalline main body} (1), this means that the $G_K$-representation $T(\E)[1/\pi]$ is $E$-crystalline, as desired.
\end{proof}
\begin{lem}\label{BKF}
Fix an object $M\in \vect(\Prism_{\O_C})$ with \' etale realization $T:=T(M)$. Then 
\begin{align}\label{4.26}
    T \otimes_{\O_E}A_{\inf,E}\left[\dfrac{1}{\mu}\right]=M\left[ \dfrac{1}{\mu}\right]
\end{align}
as submodules of $T\otimes_{\O_E}W_{\O_E}(C^\flat)=M\otimes_{A_{\inf,E}}W_{\O_E}(C^\flat)$. 
\end{lem}
\begin{proof}
The proof is similar to that of \cite[Lemma 4.26]{BMS1}. In fact, we will only explain the reduction to the case $\varphi_M^{-1}$ maps $M$ into $M$; the rest of the argument is identical to that in \textit{loc. cit.} For this, we need a suitable variant of the usual Tate twist. Let $A_{\inf,E}\{1\}\in \vect(\Prism_{\O_C})$ be the rank one object with a basis $e$, and $\varphi_q(e)=\tfrac{1}{\widetilde{\xi}}e$. For an integer $n$, we set $M\{n\}:=M\otimes A_{\inf,E}\{1\}^{\otimes n}$. The \' etale realization of $A_{\inf,E}\{1\}$ is 
    \begin{align*}
        \{xe\;|\;x\in W_{\O_E}(C^\flat), \varphi_q(x)=\widetilde{\xi}x\}=\{xe\;|\;\varphi_q(x/\mu)=x/\mu\}
        =\O_E\mu e.
    \end{align*}
    In particular, we see that the module $A_{\inf,E}\{n\}$ satisfies \ref{4.26} for any integer $n$. Now pick $n\gg 0$ so that $\varphi_M^{-1}(M)\subseteq \tfrac{1}{\widetilde{\xi}^{n}}M$. Then $\varphi_{M}^{-1}(M\{n\})\subseteq M\{n\}$, and we can replace $M$ by this $M\{n\}$ (as the functor $M\mapsto T(M)$ is tensor-compatible in $M$).
\end{proof}
\subsection{Full faithfulness}\label{full faithfulness of etale realization}
We now prove the full faithfulness of the \' etale realization function in Theorem \ref{main thm body}. With everything in place, we can follow the proof of the corresponding statement in \cite{prismatic}.
\begin{proof}[Proof of full faithfulness in Theorem \ref{main thm body}]
Fix a Breuil--Kisin prism $(\fkS_E,I)$ in $(\O_K)_{\Prism,\O_E}$; this gives a cover of the final object in the associated topos. In particular, faithfulness of the \' etale realization functor $T$ on $X$ reduces to the analogous statement over $(\fkS_E,I)$, which in turn follows from injectivity of the map $\fkS_E\to \fkS_E[1/I]^{\wedge}_{\pi}$. The same argument also shows that $T$ is faithful over $\fkS_E^{(1)}$\footnote{By construction of $\fkS_E^{(1)}$ as a suitable prismatic envelope (namely $\fkS_E^{(1)}=W_{\O_E}(k)[[u,v]]\{(u-v)/E(u)\}^{\wedge}_{(\pi,E(u))}$), $\fkS_E^{(1)}$ is $(\pi,I)$-completely flat over $\fkS_E$ (via either structure map), cf. \cite[Proposition 2.6.6]{ito2023prismatic}. As $\fkS_E$ is transversal, the same holds for $\fkS_E^{(1)}$ (see e.g. \cite[Lemma 4.7]{BMSII}). In fact, as $\fkS_E$ is Noetherian, $\fkS_E^{(1)}$ is even classically flat over $\fkS_E$ (see e.g. \cite[Proposition 2.2.4]{EG22}), but we will not need this.}, the self-coproduct of $\fkS_E$ with itself in $(\O_K)_{\Prism,\O_E}$. 

Fullness of $T$ now reduces formally to the analogous statement over $\fkS_E$, i.e. fullness of the base change functor 
\begin{displaymath}
    \vect(\fkS_E)\to \vect(\fkS_E[1/I]^{\wedge}_{\pi}),
\end{displaymath}
which is Proposition \ref{fully faithful etale realization Kisin}. (Indeed, given prismatic $F$-crystals $\M$ and $\N$ and a map $g: \M[1/I]^{\wedge}_{\pi}\to \N[1/\I]^{\wedge}_{\pi}$ in $\mathrm{Vect}^{\varphi_q}(X_{\Prism},\O_{\Prism}[1/\I_{\Prism}]^{\wedge}_{\pi})$, we need to extend $g$ to a map $\M\to \N$. First, by fullness over $\fkS_E$ (i.e. Proposition \ref{fully faithful etale realization Kisin}), the map $g(\fkS_E)$
extends (necessarily uniquely) to a map $f': \M(\fkS_E)\to \N(\fkS_E)$. The map $f'$ induces two a priori distinct maps
$a; b: \M(\fkS_E^{(1)})\to \N(\fkS_E^{(1)})$ via extending along either structure map, and we need to show that $a = b$ (for then $f'$ will lift to the desired map $\M\to \N$). But by faithfulness over $\fkS_E^{(1)}$, it suffices to show that $a[1/I]^{\wedge}_{\pi}=b[1/I]^{\wedge}_{\pi}$, which is true as both are equal to $g(\fkS_E^{(1)})$.)
\end{proof}
The following lemma was used above. 
\begin{lem}
    Let $(A,I)$ be a transversal $\O_E$-prism. The natural map $A\to A[1/I]^{\wedge}_{\pi}$ is injective.
\end{lem}
 Here, following the terminology in \cite{prismaticdieudonne}, an $\O_E$-prism $(A,I)$ is called transversal if $A/I$ is $\pi$-torsion free; in this case $A$ is itself $\pi$-torsion free. 
 \begin{proof}
     As the source is $\pi$-adically separated and the target is $\pi$-torsion free, it suffices to show that the map is injective modulo $\pi$, which follows directly from transversality of $(A,I)$.
 \end{proof}
 \begin{remark}\label{full faithfulness via Fargues}
     As in \cite{prismatic}, instead of working with a Breuil--Kisin prism, one can also work directly with the prism $\Prism_{\O_C}$, and deduce the desired full faithfulness from (the easy direction of) Fargues' classification of $F$-crystals over $\vect(\Prism_{\O_C})$.
 \end{remark}
\subsection{Essential surjectivity}\label{essential surjectivity}
The goal of this subsection is to prove the essential surjectivity part of Theorem \ref{main thm body}. As explained in the introduction, the general strategy of our proof follows that of \cite[\textsection 6]{prismatic}; the main difference is that instead of using inputs from \cite{AMMNBeilinson} to prove the desired boundedness of descent data (Proposition 6.10 in \cite{prismatic}), we adapt a key lemma from \cite{duliu}, which will in fact allows us to prove a more general result (Proposition \ref{boundedness descent data}). 

We first collect some further ring-theoretic properties on transversal prisms, which will be used repeatedly in what follows.
\begin{lem}\label{ring theoretic facts}
Let $(A,(d))$ be a transversal $\O_E$-prism.
\begin{itemize}
    \item[(1)] The ring $A\langle d/\pi\rangle [1/\pi]$ is $d$-torsion free and $d$-adically separated.
    \item[(2)] $A\cap d^n A\langle d/\pi\rangle [1/\pi]=d^nA$ and $A\langle d/\pi\rangle\cap d^n A\langle d/\pi\rangle [1/\pi]=(d/\pi)^nA\langle d/\pi\rangle$ for each $n\geq 0$. 
    \item[(3)] For each $n\geq 0$, the natural map 
    \begin{displaymath}
        A[1/\pi]^{\wedge}_d\to A\langle \varphi_q^n(d)/\pi\rangle [1/\pi]^{\wedge}_d
    \end{displaymath}
    is an isomorphism. Moreover, the natural maps $A\to A\langle \varphi^n(d)/\pi\rangle[1/\pi]\to A\langle d/\pi\rangle [1/\pi]$ are injective.
\end{itemize}
\end{lem}
As the proof shows, parts (1) and (2) in fact hold for any pair $(A,d)$ such that: (i) $(\pi,d)$ forms a regular sequence, (ii) $A$ is (classically) $(\pi,d)$-complete; moreover part (3) needs additionally only the underlying $\delta_E$-ring structure (rather than the full prism structure) of $A$.  
\begin{proof}
(1) Recall that $A\langle d/\pi\rangle$ is by definition the $\pi$-adic completion of the $A$-subalgebra $A[d/\pi]$ of $A[1/\pi]$ generated by $d/\pi$. We note in particular that it is $\pi$-torsion free. We claim that the natural map $A[x]/(\pi x-d)\to A[d/\pi], x\mapsto d/\pi$ is an isomorphism. As it is clearly an isomorphism after inverting $\pi$, it suffices to show that the source is $\pi$-torsion free, which in turn follows formally from the facts that $A[x]$ is $\pi$-torsion free and that $\pi x-d=-d$ is regular on $(A/\pi)[x]$. Thus
\begin{displaymath}
    A\langle d/\pi\rangle= A[x]/(\pi x-d)^{\wedge}_p=A\langle x\rangle/J,
\end{displaymath}
where $A\langle x\rangle :=A[x]^{\wedge}_p$ and $J:=\overline{(\pi x-d)}=\cap_{n\geq 0}(\pi x-d,\pi^n)A\langle x\rangle$ is the closure of $(\pi x-d)$ in $A\langle x\rangle$ for the $\pi$-adic topology. We need to show that if $f\in A\langle x\rangle$ satisfies $df\in (\pi x-d,\pi^n)$ for all $n\gg 0$, then the same holds for $f$. Write $df=(\pi x-d)g+\pi^n h$ for some $g,h\in A\langle x\rangle$. As $d(f+g)=\pi(xg+\pi^{n-1}h)$, we can write $f+g=\pi k, xg+\pi^{n-1}h=dk$ for some $k\in A\langle x\rangle$. Now as $A\langle x\rangle /(x,\pi^{n-1})=A[x]/(x,\pi^{n-1})=A/\pi^{n-1}$ is $d$-torsion free, $k\in (x,\pi^{n-1})$, say $k=xa+\pi^{n-1}b$ for some $a,b\in A\langle x\rangle$. Then $x(g-da)=\pi^{n-1}(db-h)$ whence $g-da=\pi^{n-1}a'$ for some $a'\in A\langle x\rangle$ whence $f=\pi k-g=\pi(xa+\pi^{n-1}b)-(da+\pi^{n-1}a')=(\pi x-d)a+\pi^{n-1}(\pi b-a')\in (\pi x-d,\pi^{n-1})$, as wanted.

For the last statement of (1), it suffices to show that
\begin{displaymath}
A\langle d/\pi\rangle\cap \cap_{n\geq 0}d^n A\langle d/\pi\rangle [1/\pi]=0.    
\end{displaymath}
We first check that $A\langle d/\pi\rangle\cap (d/\pi)^nA\langle d/\pi\rangle [1/\pi]=(d/\pi)^n A\langle d/\pi\rangle$. As $d$ (hence $d/\pi$) is regular on $A\langle d/\pi\rangle$ by (1), we may assume $n=1$. We need to show that $A\langle d/\pi\rangle/(d/\pi)$ is $\pi$-torsion free. We claim that $A\langle d/\pi\rangle /(d/\pi)\simeq A/d$ via the natural map. Indeed, the map $A\langle x\rangle \twoheadrightarrow A, x\mapsto 0$ maps $\overline{(\pi x-d)}$ into $dA\subseteq A$ (as $dA$ is closed for the $\pi$-topology on $A$: $A/dA=\mathrm{coker}(A\xrightarrow{\times d}A)$ is derived $\pi$-complete, hence classically $\pi$-complete as it is $\pi$-torsion free by our assumptions on $A$). We thus get a surjection $A\langle x\rangle/(\overline{(\pi x-d}),x)\twoheadrightarrow A/d, x\mapsto 0$. But $A/d=A\langle x\rangle/(\pi x-d,x)$ also surjects naturally onto $A\langle x\rangle/(\overline{(\pi x-d}),x)$ whence $A/d=A\langle x\rangle/(\overline{(\pi x-d}),x)$:
\begin{displaymath}
    \begin{tikzcd}
        A/d\ar[rr, bend right=30, equal] \ar[r,twoheadrightarrow] & A\langle x\rangle/(\overline{(\pi x-d}),x)\ar[r,twoheadrightarrow] & A/d.\\
    \end{tikzcd}
\end{displaymath}
Thus, $A\langle d/\pi\rangle/(d/\pi)=A/d$ is $\pi$-torsion free, as wanted. Note that the equality $A\langle d/\pi\rangle/(d/\pi)=A/d$ also implies that $A\cap (d/\pi) A\langle d/\pi\rangle=dA$. By induction (using that $(\pi,d)$ is a regular sequence), we see that in fact $A\cap (d/\pi)^nA\langle d/\pi\rangle=d^nA$ for all $n\geq 0$. 

We thus need to show that $A\langle d/\pi\rangle$ is $(d/\pi)$-adically separated. Write $\overline{x}:=d/\pi$. Assume $f\in \cap_{n\geq 0}\overline{x}^nA\langle \overline{x}\rangle$. We can write $f=a_0+a_1\overline{x}+\ldots$ for a (not necessarily unique) sequence $(a_n)$ in $A$ with $a_n\to 0$ $\pi$-adically. In particular, $a_0\in \overline{x}A\langle \overline{x}\rangle\cap A =dA$ (by the preceding paragraph), say $a_0=da_0'=\pi a_0'\overline{x}$. As $\overline{x}$ is regular in $A\langle \overline{x}\rangle$, we still have $f/\overline{x}=(\pi a_0'+a_1)+\ldots\in \cap_{n\geq 0}(\overline{x}^n)$. Similarly we find that $a_1+\pi a_0'\in (\overline{x})\cap A=dA$, say $a_1+\pi a_0'=da_1'$. Next, we have $0=(\pi\overline{x}-d)(a_0'+a_1'\overline{x})$, so $a_0+a_1\overline{x}=da_0'+(da_1'-\pi a_0')\overline{x}=\pi a_1'\overline{x}^2$. Again, as $\overline{x}$ is regular, we deduce that $a_2+\pi a_1'\in dA$ say $a_2+\pi a_1'=da_2'$. Repeating this argument, we can inductively find a sequence $(a_n')$ in $A$ such that $a_n=da_n'-\pi a_{n-1}'$ for all $n\geq 0$. We claim that the sequence $(a_n')$ also tends to $0$ $\pi $-adically. Once this is done, $a_0'+a_1'\overline{x}+\ldots$ makes sense as an element in $A\langle\overline{x}\rangle$ and we have $f=a_0+a_1\overline{x}+\ldots=(d-\pi\overline{x})(a_0'+a_1'\overline{x}+\ldots)=0$, as wanted. We will show by induction on $n$ that $a_m'\in \pi^nA$ for all $m\gg 0$ (depending on $n$). So assume that $a_m'\in \pi^nA$ for all $m\geq m_0$. Enlarging $m_0$ if necessary, we may also assume that $a_m\in \pi^{n+1}A$ for all $m> m_0$ (as $a_m\to 0$ $\pi$-adically by assumption). Then $da_m'=a_m+\pi a_{m-1}'\in \pi^{n+1}A$, and so, as $(\pi,d)$ is a regular sequence, $a_m'\in \pi^{n+1}A$ for all $m>m_0$, as claimed.

(2) This is contained in the proof of (1) above.

(3) For the first statement, as the sources and target are both $d$-complete and $d$-torison free (by part (1) as $A\langle \varphi_q^n(d)/\pi\rangle =A\langle d^{q^n}/\pi\rangle$), it suffices to show that the map 
\begin{displaymath}
    A[1/\pi]/d^{q^n}\to A\langle d^{q^n}/\pi\rangle[1/\pi]/d^{q^n}
\end{displaymath}
is an isomorphism, which in turn follows from the proof of (1). The second statement can be proved similarly (i.e. by reducing modulo $d$ (or $d^{q^n}$)).
\end{proof}
The following variant of \cite[Lemma 2.2.10]{duliu} records a simple effect of $\varphi_q$ on the $d$-adic filtration of $A\langle d/\pi\rangle$. 
\begin{lem}\label{lemma phi(Fil) Du-Liu}
    Let $(A,(d))$ be a transversal $\O_E$-prism. Then given any $h\geq 0$, 
    \begin{displaymath}
        \varphi_q(d^m A\langle d/\pi\rangle)\subseteq A+d^{m+h}A\langle d/\pi\rangle
    \end{displaymath}
    for all $m\gg 0$ (depending only on $h$).
\end{lem}
\begin{proof}
We will show more generally that
\begin{displaymath}
    \varphi_q(d^m A\langle d/\pi\rangle)\subseteq A+\dfrac{d^{q(m+1)}}{\pi}A\langle d^q/\pi\rangle\quad\text{for all $m\geq 0$}.
\end{displaymath}
This clearly implies the lemma. Write $\varphi_q(d)=d^q+\pi a$ for some $a\in A$. As $\varphi_q(A\langle d/\pi\rangle)\subseteq A\langle d^q/\pi\rangle$, it suffices by the binomial theorem to show that
\begin{displaymath}
    d^{q(m-k)}\pi^k A\langle d^q/\pi\rangle \subseteq A+\dfrac{d^{q(m+1)}}{\pi}A\langle d^q/\pi\rangle\quad \text{for any $0\leq k\leq m$}.
\end{displaymath}
This follows immediately from the inclusion $A\langle d^q/\pi\rangle \subseteq (1/\pi^{k})A+(d^q/\pi)^{k+1}A\langle d^q/\pi\rangle$.

\end{proof}
\begin{prop}\label{crucial du-liu}
Let $(A,(d))$ be a transversal $\O_E$-prism. Let $h\geq 0$. Assume
\begin{displaymath}
 d^hY=B\varphi_q(Y)C   
\end{displaymath} 
with matrices $Y\in M_d(A\langle d/\pi\rangle [1/\pi])$ and $B, C\in M_d(A)$. Then $Y\in M_d(A[1/\pi])$.
\end{prop}
\begin{proof}
We will follow the proof of \cite[Proposition 2.2.11]{duliu}\footnote{It is not clear to us if the various rings in \cite{duliu} indeed agree with the more standard rings denoted by the same notation; for instance, we do not know if the ring $A^{(2)}_{\max}$ there equals literally to $A^{(2)}\langle I/ p\rangle$. Nevertheless, it is relatively straightforward to adapt the arguments of \textit{loc. cit.} to the present setting. (In an updated version of \cite{duliu}, this point is discussed in Remark 2.2.11.)}. Replacing $Y$ by $\pi^kY$ for some $k\gg 0$, we may assume that
\begin{displaymath}
    Y=Y_{m_0}+Y'
\end{displaymath}
for some $Y_{m_0}\in M_d(A)$ and $X\in M_d(d^{m_0}R)$, where $m_0$ satisfies $\varphi_q(d^m R)\subseteq A+d^{m+1+h}R$ for all $m\geq m_0$ ($m$ exists by the previous lemma); here $R:=A\langle d/\pi\rangle$. We will show that $Y\in M_d(A)$. The idea is to approximate $Y$ in the $d$-adic topology on $R$ by elements of $A$. More precisely, we will construct inductively a sequence $(Y_m)_{m\geq m_0}$ in $M_d(A)$ with the property that $Y_{m+1}\equiv Y_m\bmod{d^m A}$ and $Y_{m}\equiv Y\bmod{d^m R}$ for all $m\geq m_0$. As $A$ and $R$ are both $d$-adically complete (note that $R$ is $\pi$-adically complete with $\pi|d$), this implies that $Y\in M_d(A)$, as wanted.

Assume $Y_m$ has been constructed. Write $Y=Y_m+X$ with $X\in M_d(d^{m}R)$. By assumption,
\begin{displaymath}
    d^h(Y_m+X)=B\varphi_q(Y_m)C+B\varphi_q(X)C.
\end{displaymath}
By our choice of $m_0$ we can write $B\varphi_q(X)C= Z+d^{h} X'$ for some $Z\in M_d(A)$ and $X'\in M_d(d^{m+1}R)$. Then $B\varphi_q(Y_m)C+Z-d^h Y_m=d^h(X-X')$ has entries in $A\cap d^{h+m}R=d^{h+m}A$ by Lemma \ref{ring theoretic facts} (2), say $d^{h+m}X''$ with $X''\in M_d(A)$. Again, as $d$ is regular on $R$, we obtain $X-X'=d^{m}X''$. Now set $Y_{m+1}:=Y_m+d^m X''$. 
\end{proof}
\begin{remark}
Note that one cannot weaken the conditions on $B,C$ into $B,C\in M_d(A[1/\pi])$. For instance, the infinite product
\begin{displaymath}
    \lambda:=\prod_{n\geq 0}\varphi^n(E(u)/E(0))\in \O
\end{displaymath}
satisfies $\lambda=(E(u)/E(0))\varphi(\lambda)\in \fkS[1/p]\cdot \varphi(\lambda)$, but $\lambda\notin \fkS[1/p]$\footnote{For instance, take $K=\mathbf{Q}_p$ and $\pi_K=-p$. Then $E(u)=u+p$ and the coefficient of $u^{1+p+\ldots+p^{n-1}}$ in $\lambda=\prod_{n\geq 0}(1+u^{p^n}/p)$ is $1/p^n$ and of course one can make $n$ arbitrarily large.}. 
\end{remark}
\begin{prop}\label{boundedness descent data}
Let $(A,(d))$ be a transversal $\O_E$-prism. Then the base change
\begin{displaymath}
    \mathrm{Vect}^{\varphi_q}(A)[1/\pi]\to \mathrm{Vect}^{\varphi_q}(A\langle d/\pi\rangle [1/\pi])
\end{displaymath}
is fully faithful; here the source denotes the isogeny category of $\vect(A)$.
\end{prop}
\begin{proof}
    Given objects $\fM_i$ in $\vect(A)$, we need to show that any $\varphi_q$-equivariant map 
\begin{displaymath}
    \alpha: \fM_1\otimes_A A\langle d/\pi\rangle [1/\pi]\to \fM_2\otimes_A A\langle d/\pi\rangle [1/\pi]
\end{displaymath}
extends to a map $\fM_1[1/\pi]\to \fM_2[1/\pi]$. (By injectivity of $A[1/\pi]\to A\langle d/\pi\rangle [1/\pi]$ (Lemma \ref{ring theoretic facts} (3)), such an extension is necessarily uniquely and $\varphi_q$-equivariant.) By Lemma \ref{direct summand of free} below $\fM_i$ can be written as a $\varphi_q$-stable direct summand of some \textit{finite free} $\varphi_q$-module over $A$. We may thus reduce to the case $\fM_1$ and $\fM_2$ are both finite free. 

Pick an $A$-basis $e_1,\ldots,e_{d_1}$ of $\fM_1$, and let $A_1\in M_{d_1}(A[1/d])$ be the matrix giving the action of $\varphi_{\fM_1}$ on this basis, i.e. $\varphi_{\fM_1}(e_1,\ldots,e_{d_1})=(e_1,\ldots,e_{d_1})A_1$; similarly let $A_2\in M_{d_2}(A[1/d])$ be the matrix giving the action of $\varphi_{\fM_2}$ on some fixed basis of $\fM_2$. As $\alpha$ is $\varphi$-equivariant, we see that if $Y\in M_{d_1d_2}(A\langle d/\pi\rangle [1/\pi])$ denotes the matrix of $\alpha$ relative to the chosen bases, then 
\begin{displaymath}
    YA_1=A_2\varphi_q(Y).
\end{displaymath}
We need to show that $Y$ in fact has entries in $A[1/\pi]$. As $A_1$ is invertible, we can write the above equation as $d^h Y=B\varphi_q(Y)C$ for some $h\geq 0$, and matrices $B, C$ with entries in $A$. Then by Proposition \ref{crucial du-liu} below, $Y$ has entries in $A[1/\pi]$, as wanted. 
\end{proof}
The following lemma was used above.
\begin{lem}\label{direct summand of free}
Let $(A,(d))$ be an $\O_E$-prism. Any object in $\vect(A)$ can be realized as a $\varphi_q$-stable direct summand of some finite free object.
\end{lem}
\begin{proof}
The proof is similar to that of \cite[Lemma 4.3.1]{EG22}. Fix $\fM\in \vect(A)$. Pick an $A$-module $\fN$ so that $\fF:=\fM\oplus \fN$ is finite free. Then $(\fF\oplus \fN)\oplus \fM=\fF\oplus \fF$ is finite free, and by fixing an isomorphism $\varphi_F: \varphi_q^*\fF\simeq \fF$, one can endow $\fF\oplus \fN$ with a $\varphi_q$-structure via the composition
\begin{align*}
    \varphi_q^*(\fF\oplus \fN)[1/d] &\overset{\varphi_{\fF}}{\simeq}\fF[1/d]\oplus \varphi_q^*\fN[1/d]=\fM[1/d]\oplus \fN[1/d]\oplus \varphi_q^*\fN[1/d]\\ & \overset{\varphi_{\fM}^{-1}}{\simeq}\varphi_q^*\fM[1/d]\oplus \fN[1/d]\oplus \varphi_q^*\fN[1/d]=\varphi_q^*\fF\oplus \fN[1/d] \\
    &\overset{\varphi_{\fF}}{\simeq} (\fF\oplus \fN)[1/d].
\end{align*}
\end{proof}
\begin{proof}[Proof of essential surjectivity in Theorem \ref{main thm body}]
Fix $T\in \mathrm{Rep}_{\O_E}^{\mathrm{cris}}(G_K)$. Let $D\in \mf(K)$ be the weakly admissible filtered $\varphi_q$-module over $K$ corresponding to $T[1/\pi]$ under the equivalene in Theorem \ref{weakly admissible vs crys rep}. Fix a Breuil--Kisin prism $(\fkS_E,I)$ in $X_{\Prism}$. We note firstly that the functor $\M: \mf(K)\to \vect(\fkS_E\langle I/\pi\rangle [1/\pi])$ from \textsection \ref{construction kisin 06} in fact lifts to a functor
\begin{align}
    \M: \mf(K)\to \vect(X_{\Prism},\O_{\Prism}\langle \I_{\Prism}/\pi\rangle [1/\pi])
\end{align}
by applying exactly the same construction for each object in $X_{\Prism}$ (cf. \cite[Construction 6.5]{prismatic}\footnote{Keep in mind that in defining the constant $F$-crystal $D\otimes_{W_{\O_E}(k)}\O_{\Prism}\langle \I_{\Prism}/\pi\rangle [1/\pi]$, we use the canonical $W_{\O_E}(k)$-algebra structure on objects of $(\O_K)_{\Prism,\O_E}$; see Remark \ref{algebra structure Wk}.}). Moreover, we have seen in \textsection \ref{weak admissibility and slope 0} that the weak admissibility of $D$ implies that $\M(D)$ extends to an object $\fM\in \vect(\fkS_E)$, which therefore comes equipped with a descent datum
\begin{displaymath}
    \alpha:\fM\otimes_{\fkS_E,p_1}\fkS_E^{(1)}\langle I/\pi\rangle [1/\pi]\simeq \fM\otimes_{\fkS_E,p_2}\fkS_E^{(1)}\langle I/\pi\rangle [1/\pi]
\end{displaymath}
by the last sentence. Proposition \ref{boundedness descent data} then shows that $\alpha$ extends uniquely to a descent datum
\begin{displaymath}
    \alpha: \fM\otimes_{\fkS_E,p_1}\fkS_E^{(1)}[1/\pi]\simeq \fM\otimes_{\fkS_E,p_2}\fkS_E^{(1)}[1/\pi];
\end{displaymath}
in other words, $\fM[1/\pi]$ lifts naturally to an object in $\vect(X_{\Prism},\O_{\Prism}[1/\pi])$. The rest of the arguments in \cite[\textsection 6.4]{prismatic} now carry over to our setting (applied to the object $\M':=\fM\otimes_{\fkS}\Prism_{\O_C}\in \vect(\Prism_{\O_C})$), finishing the proof\footnote{We can also argue slightly differently as follows. Namely, we first check that $T(\fM)[1/\pi]\simeq T[1/\pi]|_{G_{K_{\infty}}}$, whence by Lemma \ref{beauville laszlo glueing and extending vbs}, we may pick $\fM$ uniquely so that $T(\fM)\simeq T|_{G_{K_{\infty}}}$; here $K_{\infty}$ as usual denotes the Kummer extension $K(\pi_K^{1/q^{\infty}})$ of $K$. But as $\fM[1/\pi]$ lifts to an $F$-crystal on $\O_K$, its ``\' etale realization'' $T(\fM)[1/\pi]\in \mathrm{Vect}(E)$ carries a $G_K$-action (extending the natural $G_{K_{\infty}}$-action), which is in fact $E$-crystalline by exactly the same argument as in the proof of Theorem \ref{main thm body}. By full faithfulness of the restriction functor $\mathrm{Rep}_E^{\mathrm{cris}}(G_K)\to \mathrm{Rep}_E(G_{K_{\infty}})$, we deduce that $T(\fM)[1/\pi]\simeq T[1/\pi]$ $G_K$-equivariantly. Thus, we have arranged so that $\fM\otimes \Prism_{\O_C}[1/I]^{\wedge}_{\pi}$ is stable under the $G_K$-action on $\fM\otimes \Prism_{\O_C}[1/I]^{\wedge}_{\pi}[1/\pi]$ (coming from the descent datum $\alpha$). Now we can simply follow the arguments in the fourth paragraph of \cite[\textsection 6.4]{prismatic} to conclude (the difference is that we do not need to run another modification as in \textit{loc. cit.}: the resulting prismatic $F$-crystal has \' etale realization $T$ on the nose).

}. 
\end{proof}
\subsection{Relation with Kisin--Ren's theory \texorpdfstring{\cite{KisinRen}}{}}\label{relation kisin--ren}
In this subsection, we show that Theorem \ref{main thm body} encodes the classification of Galois stable lattices in $E$-crystalline representations in \cite{KisinRen} upon specializing to a suitable prism in $(\O_K)_{\Prism}$, in the same way that it encodes the theory of Breuil--Kisin theory in \cite{Kis06} by specializing to a Breuil--Kisin prism.

Let $\G$ be the Lubin--Tate formal $\O_E$-module over $\O_E$ corresponding to a uniformizer $\pi\in E$. Pick a coordinate $X$ for $\G$, i.e. an isomorphism $G\simeq \Spf(\O_E[[X]])$. For $a\in \O_E$, denote by $[a]\in \O_E[[X]]$ the power series giving the action of $a$ on $\G$. Let $K_{\infty}\subseteq \overline{K}$ be the subfield generated by the $\pi$-power torsion points of $\G$ and write $\Gamma:=\mathrm{Gal}(K_{\infty}/K)$. The Tate module $T_{p}\G$ is a free $\O_E$-module of rank one, and the action of $G_K$ on $T_p\G$ is given by the Lubin--Tate character $\chi: G_K\to \O_E^\times$. 

Fix a generator $v=(v_n)_{n\geq 0}$ of $T_{p}\G$. As in \cite{KisinRen}, we will assume in what follows that $K\subseteq K_{0,E}K_{\infty}$. Fix $m\geq 1$ so that $K\subseteq K_{0,L}(v_m)$. Let $Q(u):=[\pi^m](u)/[\pi^{m-1}](u)$. As before, write $\fkS_E:=W_{\O_E}(k)[[u]]$; however we now equip $\fkS_E$ with a $\varphi_q$-action and a $\Gamma$-action given respectively by $\varphi_q(u):=[\pi](u)$ and $\gamma(u)=[\chi(\gamma)](u)$ for $\gamma\in \Gamma$.

It is easy to check that the preceding $\varphi_q$-action makes the pair $(\fkS_E,(Q(u)))$ into an $\O_E$-prism. Moreover, as the map $\fkS_E\twoheadrightarrow \O_{K_{0,E}(v_m)}, u\mapsto v_m$ is surjective with kernel $(Q(u))$, our assumption on $m$ gives a map $\O_K\to \fkS_E/(Q(u))$, making $(\fkS_E,(Q(u))$ into an $\Gamma$-equivariant object of $(\O_K)_{\Prism,\O_E}$, which we will denote by $\fkS_E'$ to distinguish with the Breuil--Kisin prism $\fkS_E$ introduced earlier. Note also that mapping $u\mapsto \varphi_q^{-(m-1)}([v]_{\G})$ defines a $G_K$-equivariant map $\fkS_E'\to \Prism_{\O_C}$ in $(\O_K)_{\Prism,\O_E}$. In particular, $\fkS_E'$ again gives a cover of the final object in the associated topos. 
\begin{defn}
Let $\mathrm{Vect}^{\varphi_q,\Gamma}(\fkS_E')$ denote the category of $\fM\in \vect(\fkS_E')$ equipped with a semilinear action of $\Gamma$ which commutes with $\varphi_q$, and such that $\Gamma$ acts trivially on $\fM/u\fM$. Inside this, we have a full subcategory $\mathrm{Vect}^{\varphi_q,\Gamma,\mathrm{an}}(\fkS_E')$ consisting of objects for which the $\Gamma$-action is analytic in a suitable sense; see \cite[\textsection (2.1.3)]{KisinRen} and \cite[\textsection (2.4.3)]{KisinRen}. 
\end{defn}
\begin{prop}
    Consider the functor 
    \begin{displaymath}
        D_{\fkS_E'}: \mathrm{Rep}_{\O_E}^{\mathrm{cris}}(G_K)\to \mathrm{Vect}^{\varphi_q,\Gamma}(\fkS_E')
    \end{displaymath}
    defined by composing the inverse of the equivalence in Theorem \ref{main thm body} with the evaluation  at $(\fkS_E',(Q(u)))\in (\O_K)_{\Prism,\O_E}$. Then $D_{\fkS_E'}$ is fully faithful with essential image $\mathrm{Vect}^{\varphi_q,\Gamma,\mathrm{an}}(\fkS_E')$.
\end{prop}
\begin{proof}
We first check that $D_{\fkS_E'}$ is well-defined, i.e. given any object $\E\in \vect((\O_K)_{\Prism},\O_{\Prism})$, the value $\E(\fkS_E')$ is naturally an object in $\mathrm{Vect}^{\varphi_q,\Gamma}(\fkS_E')$. As $\fkS_E'$ is a $\Gamma$-equivariant object in $(\O_K)_{\Prism,\O_E}$, $\fM:=\E(\fkS_E')$ carries a natural $\Gamma$-action. Now the map $\fkS_E'\to \Prism_{\O_C}$ induces by reducing modulo $u$ a $G_K$-equivariant map
\begin{displaymath}
    (W_{\O_E}(k),(\pi))\to (W_{\O_E}(\overline{k}),(\pi))
\end{displaymath}
in $(\O_K)_{\Prism,\O_E}$; as $W_{\O_E}(k)$ is fixed by the natural $G_K$-action on $W_{\O_E}(\overline{k})$, the crystal property of $\E$ again implies that $\Gamma$ acts trivially on $\E(W_{\O_E}(k))\simeq \fM/u\fM$, as wanted. 

Next, we show that $D_{\fkS_E'}$ lands in the subcategory of analytic objects; by the very definition of the latter (\cite[\textsection (2.4.3)]{KisinRen}), it suffices to prove the analogous statement for the composition 
\begin{align*}
      \mathrm{Rep}_{\O_E}^{\mathrm{cris}}(G_K)[1/\pi]\overset{D_{\fkS_E'}[1/\pi]}{\to} \mathrm{Vect}^{\varphi_q,\Gamma}(\fkS_E')[1/\pi]\to \mathrm{Vect}^{\varphi_q,\Gamma}(\fkS_E'\langle I/\pi\rangle[1/\pi])\simeq \mathrm{Vect}^{\varphi_q,\Gamma}(\O'),
\end{align*}
where $\O'$ again denotes the ring of functions on the rigid open unit disk over $K_{0,E}$, but now equipped with the Frobenius $u\mapsto [\pi](u)$. By unwinding definitions, this coincides (upon identifying $\mathrm{Rep}^{\mathrm{cris}}_{E}(G_K)\simeq \mathrm{MF}^{\varphi_q,w.a}(K)$) with the functor 
\begin{displaymath}
    \M'(\cdot): \mathrm{MF}^{\varphi_q,w.a.}(K)\to \mathrm{Vect}^{\varphi_q,\Gamma}(\O')
\end{displaymath}
from \cite[\textsection (2.2)]{KisinRen} (cf. Remark \ref{compare with Kisin}); in particular, we know from Lemma (2.2.1) of \textit{loc. cit.} that it indeed factors through the subcategory of analytic objects.

Consider now the composition
\begin{displaymath}
\mathrm{Rep}_{\O_E}^{\mathrm{cris}}(G_K)\overset{D_{\fkS_E'}}{\to} \mathrm{Vect}^{\varphi_q,\Gamma,\mathrm{an}}(\fkS_E')\overset{T}{\to} \mathrm{Rep}_{\O_E}(G_K),  
\end{displaymath}
where $T$ again denotes the \' etale realization functor $\fM\mapsto (\fM\otimes_{\fkS_E'}W(\mathbf{C}^\flat))^{\varphi_q=1}$. Unwinding again the construction of $D_{\fkS_E'}$, we see that this is nothing but the forgetful functor. Moreover, by \cite[Corollary (3.3.8)]{KisinRen}, $T$ defines an equivalence onto the subcategory $\mathrm{Rep}_{\O_E}^{\mathrm{cris}}(G_K)$. It follows that $D_{\fkS_E'}$ is also an equivalence, as wanted.
\end{proof}
\subsection{Relation with \texorpdfstring{$\pi$}{pi}-divisible \texorpdfstring{$\O_E$}{OE}-modules over \texorpdfstring{$\O_K$}{OK}}
In this subsection, we combine Theorem \ref{main thm body} with a key result on \textit{minuscule} prismatic $F$-crystals from \cite{prismaticdieudonne} to deduce a classification result for $\pi$-divisible $\O_E$-modules over $\O_K$ (Theorem \ref{classification pi divisible Cheng}).
\begin{defn}[Minuscule Breuil--Kisin modules]
Let $(A,I)$ be an $\O_E$-prism. An object $M\in \vect(A,I)$ is called minuscule (or effective of height 1) if $\varphi_M$ is induced by a map $\varphi_q^*M\to M$ with cokernel killed by $I$. An object $\E\in \vect(X_{\Prism},\O_{\Prism})$ is called minusucle if for all $(A,I)\in X_{\Prism}$, the value $\E(A)$ is minuscule. Following \cite{prismaticdieudonne}, we denote the resulting categories by $\mathrm{BK}_{\min}(A,I)$ and $\mathrm{DM}(X)$, respectively. 
\end{defn}
\begin{prop}[{\cite{prismaticdieudonne},\cite{ito2023prismatic}}]\label{breuil kisin eqvaluate equivalence}
    Fix a Breuil--Kisin prism $(\fkS_E,I)\in (\O_K)_{\Prism,\O_E}$. Then evaluation at $\fkS_E$ defines an equivalence 
    \begin{displaymath}
        \mathrm{DM}(\O_K)\simeq \mathrm{BK}_{\min}(\fkS_E).
    \end{displaymath}
\end{prop}
\begin{proof}
    This is proved in \cite[Theorem 5.12]{prismaticdieudonne} in case $E=\mathbf{Q}_p$ (and for a more general class of rings in place of $\O_K$). The case of general $E$ is then proved in \cite[Proposition 7.1.1]{ito2023prismatic}.
\end{proof}
\begin{thm}[Fontaine, Kisin, Raynaud, Tate]\label{liu}
Sending a $p$-divisible group to its $p$-adic Tate module defines an equivalence 
\begin{align*}
    \mathrm{BT}(\O_K) &\simeq \mathrm{Rep}_{\mathbf{Z}_p}^{\mathrm{cris},\{0,1\}}(G_K)\\
    G &\mapsto T_p(G);
\end{align*}
here the source denotes the category of $p$-divisible groups over $\O_K$, and the target denotes the category of lattices in crystalline $G_K$-representations with Hodge--Tate weights in $\{0,1\}$.
\end{thm}
\begin{proof}
 This is well-known; see e.g. \cite[Theorem 2.2.1]{Liupdivisible}.
\end{proof}
\begin{defn}[cf. \cite{Faltings02}]
Let $R$ be an $\O_E$-algebra. A $\pi$-divisible $\O_E$-module over $R$ is a $p$-divisible group $G$ over $R$ together with an action $\O_E\to \mathrm{End}(G)$, which is strict in the sense that the induced action of $\O_E$ on $\mathrm{Lie}(G)$ agrees with the action through the structure map $\O_E\to R$. 
\end{defn}
\begin{lem}\label{pi divisible and analytic}
Let $G$ be a $p$-divisible group over $\O_K$. Then an action $\O_E\to \mathrm{End}(G)$ makes $G$ into a $\pi$-divisible $\O_E$-module over $\O_K$ if and only if the $E$-representation $V_p(G)$ is $E$-crystalline in the sense of Definition \ref{analytic definition}. 
\end{lem}
\begin{proof}
This follows from the Hodge--Tate decomposition for $G$:
\begin{align}\label{Hodge--Tate decom}
    V_p(G)\otimes_{\mathbf{Q}_p}C\simeq (\mathrm{Lie}(G^{\vee})^{\vee}\otimes_{\O_K}C)\oplus (\mathrm{Lie}(G)\otimes_{\O_K}C(1)).
\end{align} 
More precisely, as $E\otimes_{\mathbf{Q}_p}C\simeq \prod_{\tau: E\hookrightarrow C}C$, an $(E\otimes_{\mathbf{Q}_p}C)$-module $V$ always decomposes uniquely as $V=\bigoplus_{\tau}V_{\tau}$. Concretely, $V_{\tau}\subseteq V$ is the $\mathbf{C}$-subspace on which $E$ acts via $\tau: E\hookrightarrow C$; in particular, $E$ acts on $V$ via $\tau_0$ if and only $\bigoplus_{\tau\ne \tau_0}V_{\tau}=0$. Now as \eqref{Hodge--Tate decom} is functorial in $G$, it must respect the $E$-action on both sides. In particular,
\begin{displaymath}
    \bigoplus_{\tau\ne \tau_0}(V_p(G)\otimes_{E,\tau}C)\simeq \bigoplus_{\tau\ne \tau_0}(\mathrm{Lie}(G^{\vee})^{\vee}\otimes_{\O_K}C)_{\tau}\oplus\bigoplus_{\tau\ne \tau_0}(\mathrm{Lie}(G)\otimes_{\O_K}C(1))_{\tau}.
\end{displaymath}
We now have 
\begin{align*}
& \text{the action $\O_E\to \mathrm{End}(G)$ is strict}\\
   &\Longleftrightarrow \text{$\O_E$ acts on $\mathrm{Lie}(G)$ via $\tau_0: E\hookrightarrow K$}\\
   &\Longleftrightarrow \text{$E$ acts on $\mathrm{Lie}(G)\otimes_{\O_K}C(1)$ via $\tau_0: E\hookrightarrow K\subseteq C$ ($\mathrm{Lie}(G)\hookrightarrow \mathrm{Lie}(G)\otimes_{\O_K}C$ as $\mathrm{Lie}(G)$ is $\O_K$-free)}\\
   &\Longleftrightarrow \text{$\bigoplus_{\tau\ne \tau_0}(\mathrm{Lie}(G)\otimes_{\O_K}C(1))_{\tau}=0$}\\
   &\Longleftrightarrow \text{$\bigoplus_{\tau\ne \tau_0}(\mathrm{Lie}(G)\otimes_{\O_K}C(1))_{\tau}$ is trivial as a $C$-rep. (as it is a subrep. of some ${C}(1)^{\oplus n}$)}\\
   &\Longleftrightarrow \text{$\bigoplus_{\tau\ne \tau_0}(V_p(G)\otimes_{E,\tau}C)$ is trivial as a $C$-rep. (as $\mathrm{Lie}(G^{\vee})^{\vee}\otimes_{\O_K}C$ is already trivial)}\\
   &\Longleftrightarrow \text{the $E$-rep. $V_p(G)$ is $E$-crystalline},
\end{align*}
as desired.
\end{proof}
Combining Theorem \ref{main thm body}, Proposition \ref{breuil kisin eqvaluate equivalence}, Theorem \ref{liu}, and Lemma \ref{pi divisible and analytic}, we obtain the following classification of $\pi$-divisible $\O_E$-modules over $\O_K$ (including the case $p=2$).
\begin{thm}[{cf. \cite[Theorem 1.1]{chengpdivisible}, \cite[Theorem 1.0.3]{CaisLiu}}]\label{classification pi divisible Cheng}
    There is a natural equivalence between the category of $\pi$-divisible $\O_E$-modules over $\O_K$ and the category of minuscule Breuil--Kisin modules over $\fkS_E$.
\end{thm}
\begin{remark}
In \cite{CaisLiu}, the authors have obtained (by a different approach) a similar equivalence even for a large class of Frobenius lifts. In their result, the relevant category of $p$-divisible groups over $\O_K$ is formed by those which come equipped with an action of $\O_E$ \textit{for which} the rational Tate module is an $E$-crystalline representation. However, as far as we understand, the fact that this is in fact identified with the category of $\pi$-divisible $\O_E$-modules over $\O_K$ (Lemma \ref{pi divisible and analytic}) was not observed by them.
\end{remark}
\appendix
\section{\texorpdfstring{$E$}{E}-crystalline representations and filtered isocrystals}\label{appendix crystalline}
In this appendix, we prove Theorem \ref{characterization of E-crystalline main body}, thereby giving equivalent characterizations for the category of $E$-crystalline Galois representations. We will follow closely \cite[Chapitre 10]{FFCurves}, which treats the case $E=\mathbf{Q}_p$.
\subsection{Recap on the Fargues--Fontaine curve}
Let $X_E$ be the Fargues--Fontaine curve associated to $E$ and the perfectoid $\mathbf{F}_q$-algebra $F:=C^\flat$, where the $\mathbf{F}_q$-algebra structure on $F$ is defined using the fixed inclusion $\tau_0: E\hookrightarrow K\subseteq C$. As $G_K$ acts naturally on $C$ (hence on $F=C^\flat$), we obtain an induced ${E}$-linear action of $G_K$ on $X_E$. Recall that we also have a canonical identification $X_E= X_{\mathbf{Q}_p}\otimes_{\mathbf{Q}_p}E$ (cf. \cite[Th\' eor\` eme 6.5.2 (2)]{FFCurves}). Under this identification, $g\in G_K$ acts on $X_E$ as $g\otimes \mathrm{id}_E$.

Let $\pi: X_E\to X_{\mathbf{Q}_p}$ be the projection; this is a $G_K$-equivariant finite \' etale covering of degree $[E:\mathbf{Q}_p]$.
Let $\infty\in |X_{\mathbf{Q}_p}|$ be the distinguished closed point corresponding to the tautological untilt $\mathbf{Q}_p\hookrightarrow C$ of $F$. Recall that $\infty$ is fixed by $G_K$, and in fact the unique closed point of $X_{\mathbf{Q}_p}$ whose $G_K$-orbit is finite (see \cite[Proposition 10.1.1]{FFCurves}). For each $\tau: E\hookrightarrow C$, let $\infty_{\tau}$ be the closed point in $X_E $ corresponding to the (Frobenius isomorphism class of the) untilt $\tau: E\hookrightarrow C$. Concretely, $\infty_{\tau}$ is given by the (image of the) closed immersion $\Sp(C)\xrightarrow{(\infty,\tau)}X_{\mathbf{Q}_p}\otimes_{\mathbf{Q}_p}E=X_E$. 
\begin{lem}
The assignment $\tau\mapsto \infty_{\tau}$ defines a $G_K$-equivariant bijection $\mathrm{Hom}_{\mathbf{Q}_p}(E,C)\simeq \pi^{-1}(\infty)$
\end{lem}
\begin{proof}
From the previous description of $\infty_{\tau}$, we see easily that the map is $G_K$-equivariant. As $\pi$ is finite \' etale of degree $[E:\mathbf{Q}_p]$ and $\infty\in X_{\mathbf{Q}_p}$ is a closed point with algebraically closed residue field (namely $C$), $\pi^{-1}(\infty)\subseteq |X_E|$ is a finite set of $[E:\mathbf{Q}_p]$ closed points. In particular, it suffices to show that the map is surjective. Indeed, a point $x\in \pi^{-1}(\infty)$ necessarily has residue field $k(x)=C$, and it is clear from the construction that $x=\infty_{\tau}$, where $\tau: E\hookrightarrow C$ is given by the composition $\Sp(k(x))\hookrightarrow X_E\to \Sp(E)$.
\end{proof}
In particular, we see that the point $\infty_{\tau_0}$ (given by the fixed embedding $\tau_0$) is fixed by $G_K$.

Recall that we have a canonical identification $\widehat{(\O_{X_{\mathbf{Q}_p}})}_\infty\xrightarrow{\sim} B_{\mathrm{dR}}^+$, where $B_{\mathrm{dR}}^+$ is the usual Fontaine's period ring constructed using $C$. Also, the inclusion $\overline{K}\hookrightarrow C$ lifts uniquely to a (necessarily $G_K$-equivariant) section $\overline{K}\hookrightarrow B_{\mathrm{dR}}^+$. Now we have a canonical identification 
\begin{displaymath}
    B_{\mathrm{dR}}^+\otimes_{\mathbf{Q}_p}E\xrightarrow{\sim }\prod_{\tau: E\hookrightarrow\overline{K}}\widehat{(\O_{X_E})}_{\infty_\tau}
\end{displaymath}
(see e.g. \cite[\href{https://stacks.math.columbia.edu/tag/07N9}{Tag 07N9}]{Sta21}). Using the section $\overline{K}\hookrightarrow B_{\mathrm{dR}}^+$ above, we can rewrite the left side as
\begin{displaymath}
   B_{\mathrm{dR}}^+\otimes_{\overline{K}}(\overline{K}\otimes_{\mathbf{Q}_p}E)= \prod_{\tau: E\hookrightarrow\overline{K}}B_{\mathrm{dR}}^+.
\end{displaymath}
Thus, for each $\tau$, there is an $E$-algebra isomorphism $\widehat{(\O_{X_E})}_{\infty_\tau}\xrightarrow{\sim }B_{\mathrm{dR}}^+$, where the right side is regarded as an $E$-algebra via the composition $\tau: E\hookrightarrow\overline{K}\hookrightarrow B_{\mathrm{dR}}^+$.

Recall that for each compact interval $I\subseteq ]0,1[$, $B_{E,I}$ denotes the completion of 
\begin{displaymath}
    B_E^b:=\{\sum_{n\gg -\infty}[x_n]\pi^n\in W_{\O_E}(C^\flat)\;|\;(x_n)\;\text{bounded}\}
\end{displaymath}
with respect to the family of norms $|\cdot|_{\rho}, \rho\in I$. We then let $B_E:=\varprojlim_{I}B_{E,I}$. Similarly using the ring $B_E^{b,+}:=W_{\O_E}(\O_{C^\flat})[1/\pi]$, we can define $B_E^+$ and $B_{E,I}^+$. See \cite[Chapitre I]{FFCurves} for a more detailed discussion.
\begin{lem}\label{Galois invariants}
For each compact interval $I\subseteq ]0,1[$, 
\begin{displaymath}
 (\mathrm{Frac}(B_{E,I}))^{G_K}=K_{0,E},  
\end{displaymath}
where $K_{0,E}:=K_0\otimes_{E_0}E$. In particular, as $B_E\subseteq B_{E,I}$, we have $(\mathrm{Frac}(B_E))^{G_K}=K_{0,E}$.  
\end{lem}
\begin{proof}
For $E=\mathbf{Q}_p$, this is \cite[Proposition 10.2.7]{FFCurves} for $E=\mathbf{Q}_p$, but the same argument works also for general $E$. One can also deduce the general case from the case $E=\mathbf{Q}_p$ as follows. By Proposition 1.6.9 of \textit{loc.cit.} (and scaling $I$), it suffices to show that $(\mathrm{Frac}(B_{\mathbf{Q}_p,I}\otimes_{E_0}E))^{G_K}=K_{0,E}$. But by Proposition 10.2.7 of \textit{loc.cit.}, $\mathrm{Frac}(B_{\mathbf{Q}_p,I})\otimes_{{E}_0}E$ is already a field, so we have $\mathrm{Frac}(B_{\mathbf{Q}_p,I}\otimes_{E_0}E)=\mathrm{Frac}({B_{\mathbf{Q}_p,I}})\otimes_{{E}_0}E$, and hence $(\mathrm{Frac}(B_{\mathbf{Q}_p,I}\otimes_{E_0}E))^{G_K}=(\mathrm{Frac}(B_{\mathbf{Q}_p,I}))^{G_K}\otimes_{E_0}E=K_0\otimes_{E_0}E=K_{0,E}$, as wanted.  
\end{proof}
\subsection{Relation with filtered isocrystals}
Let $\pi$ be a fixed choice of a uniformizer of $E$, and let $\mathbf{F}_q$ denotes the residue field of $E$. As usual, we denote by $\varphi_q$ the $E$-linear $q$-Frobenius on $K_{0,E}\simeq W_{\O_E}(k)[1/\pi]$.
\begin{defn}\label{phi q filtered}
Let $\vect(K_{0,E})$ be the category of $\varphi_q$-modules (or isocrystals) $(D,\varphi_q)$ over $K_{0,E}$, i.e. finite dimensional $K_{0,E}$-vector spaces $D$ equipped with a linear isomorphism $\varphi_q^*D\xrightarrow{\sim} D$.

Let $\mf(K)$ be the category of filtered $\varphi_q$-modules over $K$, i.e. triples $(D,\varphi_q,\mathrm{Fil}^\bullet D_K)$, where $(D,\varphi_q)\in \vect(K_{0,E})$, and $\mathrm{Fil}^\bullet D_K$ is a decreasing filtration on $D_K:=D\otimes_{K_{0,E}}K$.
\end{defn}
Fix $t_E\in (B^+_E)^{\varphi_q=\pi}$ such that $V^+(t_E)=\{\infty_{\tau_0}\}$ ($t_E$ is uniquely determined up to multiplication in $E^\times$). Let $B_{e,E}:=\Gamma(X_E\setminus \{\infty_{\tau_0}\},\O_{X_E})=(B_E^+[1/t_E])^{\varphi_q=1}$; this is a PID equipped with an action of $G_K$. As in \cite[D\' efinition 10.1.2]{FFCurves}, we let $\mathrm{Rep}_{B_{e,E}}(G_K)$ denote the category of finite free $B_{e,E}$-modules equipped with a continuous semilinear action of $G_K$.
\begin{defn}
Define functors 
\begin{align*}
    D_{\mathrm{cris},E}: \mathrm{Rep}_{B_{e,E}}(G_K) & \to \vect(K_{0,E})\\
    M & \mapsto (M\otimes_{B_{e,E}}B_E^+[1/t_E])^{G_K}
\end{align*}
and 
\begin{align*}
    V_{\mathrm{cris},E}: \vect(K_{0,E}) & \to \mathrm{Rep}_{B_{e,E}}(G_K)\\
    (D,\varphi_q) & \mapsto (D\otimes_{K_{0,E}}B_E^+[1/t_E])^{\varphi_q=1}. 
\end{align*}
\end{defn}
\begin{prop}[{{\cite[Proposition 10.2.12]{FFCurves}}}]\label{adjoints V cris}
\emph{(1)} The functors $D_{\mathrm{cris},E}$ and $V_{\mathrm{cris},E}$ are well-defined, and form an adjoint pair with $V_{\mathrm{cris},E}$ being the left adjoint.

\emph{(2)} $V_{\mathrm{cris},E}$ is fully faithful, i.e. the unit 
\begin{displaymath}
    \mathrm{id}\xrightarrow{\sim } D_{\mathrm{cris},E}\circ V_{\mathrm{cris},E} 
\end{displaymath}
is an isomorphism. 

\emph{(3)} For each $M$ in $\mathrm{Rep}_{B_{e,E}}(G_K)$, the counit 
\begin{displaymath}
    V_{\mathrm{cris},E}(D_{\mathrm{cris},E}(M))\hookrightarrow M
\end{displaymath}
is an injection.
\end{prop}
\begin{proof}
This is \cite[Proposition 10.2.12]{FFCurves} for $E=\mathbf{Q}_p$. We begin by constructing a natural isomorphism $D_{\mathrm{cris},E}\circ V_{\mathrm{cris},E}\xrightarrow{\sim}\mathrm{id}$. Let $(D,\varphi_q)\in \vect(K_{0,E})$. We claim that the natural map 
\begin{align*}
     ((D\otimes_{K_{0,E}}B_E^+[1/t_E])^{\varphi_q=1} \otimes_{ B_{e,E}}B_E^+[1/t_E])^{G_K}\to D
\end{align*}
is an isomorphism (which is clearly $\varphi_q$-equivariant). As $(B_{E}^+[1/t_E])^{G_K}=K_{0,E}$ by Lemma \ref{Galois invariants}, it suffices to show that the natural map 
\begin{align}\label{v cris fully}
    (D\otimes_{K_{0,E}}B_E^+[1/t_E])^{\varphi_q=1} \otimes_{ B_{e,E}}B_E^+[1/t_E]\to D\otimes_{K_{0,E}} B_E^+[1/t_E]
\end{align}
is an isomorphism. By replacing $D$ with $D\otimes_{K_{0,E}}W_{\O_E}(\overline{k})[1/\pi]$, we may assume $k$ is algebraically closed. Then by the Dieudonn\'e--Manin theorem, we may reduce to the case $D^\vee\simeq D_{d/h}$ is isoclinic (for some $(d,h)\in \mathbf{Z}\times \mathbf{Z}_{\geq 1}$ with $(d,h)=1$). Recall that by definition, $D_{d/h}$ admits a $K_{0,E}$-basic $x_0,\ldots,x_{h-1}$ with $\varphi_q(x_i)=x_{i+1}$ for $0<i<h-1$, and $\varphi_q(x_{h-1})=\pi^d x_0$. Thus, 
\begin{align*}
    (D\otimes_{K_{0,E}}B_E^+[1/t_E])^{\varphi_q=1} & = \mathrm{Hom}_{\varphi_q}(D^\vee,B_E^+[1/t_E])=(B_E^+[1/t_E])^{\varphi_q^h=\pi^d},
\end{align*}
and so we are reduced to showing that the map
\begin{align*}
    (B_E^+[1/t_E])^{\varphi_q^h=\pi^d}\otimes_{B_{e,E}}B_E^+[1/t_E] & \to (B_E^+[1/t_E])^{\oplus h}\\
    x \otimes a & \mapsto (ax,a\varphi_q(x),...,a\varphi_q^{h-1}(x))
\end{align*} 
is bijective. As usual let $E_h\subseteq C$ denotes the unique unramified extension of $E$ of degree $h$; in particular we have $B_{E_h}^+=B_E^+$ canonically (cf. the proof of \cite[Proposition 1.6.9]{FFCurves}). Recall also that the vanishing locus of any nonzero element $t_{E_h}\in (B_E^+)^{\varphi_q^h=\pi}$ consists of a unique (closed) point: $V^+(t_{E_h})=\{\infty_{t_{E_h}}\}$ (cf. \cite[Th\' eor\` eme 6.5.2]{FFCurves}). Choose $t_{E_h}$ so that $\{\infty_{t_{E_h}}\}$ maps into $\{\infty_t\}$ under the projection map $\pi: X_{E_h}\to X_E$. Then $(B_E^+[1/t_E])^{\varphi_q^h=\pi^d}$ is a free of rank one over ${E}_h\otimes_E B_{e,E}$ with basis $t_{E_h}^d$. As $t_{E_h}$ is a unit in $B_E^+[1/t_E]$\footnote{We claim that $t_{E_h}$ divides $t_E$ in $B_E^+$. As $\varphi_q^h(t)=\pi^h t$, we can use \cite[Th\' eor\` eme 6.2.1]{FFCurves} to write $t=t_1\ldots t_h$ where $t_i\in (B_E^+)^{\varphi_q^h=\pi}$for each $i$. Then $\infty_{t_{E_h}}\in \pi^{-1}(\infty_t)=V^+(t_1)\cup \ldots \cup V^+(t_h)$, so we must have $t_{E_h}\in E_h^\times t_i$ for some $i$ (cf. Th\' eor\` eme 6.5.2 of \textit{loc.cit.}); in particular, we have $t_{E_h}|t$, as claimed.}, it suffices to show that the map 
\begin{align*}
    E_h\otimes_E B_E^+[1/t_E] & \to (B_E^+[1/t_E])^{\oplus h}\\
    x \otimes a & \mapsto (\varphi_q^i(a)x)_{0\leq i\leq h-1}
\end{align*}
is an isomorphism. This follows immediately from the analogue decomposition $E_h\otimes_E E_h\xrightarrow{\sim} \prod_{0\leq i\leq h-1} E_h$. This gives the isomorphism in part (2). Note that the argument also shows that $(D\otimes_{K_{0,E}}B_E^+[1/t_E])^{\varphi_q=1}$ is finite free over $B_{e,E}$ of rank $\dim_{K_{0,E}}(D)$; in other words, the functor $D\mapsto V_{\mathrm{cris},E}(D)$ is rank-preserving and indeed lands in $\mathrm{Rep}_{B_{e,E}}(G_K)$.

For part (3), it suffices to show that for each $M$ in $\mathrm{Rep}_{B_{e,E}}(G_K)$, the natural map 
\begin{align*}
    V_{\mathrm{cris},E}(D_{\mathrm{cris},E}(M))= ((M\otimes_{B_{e,E}} B_E^+[1/t_E])^{G_K}\otimes_{K_{0,E}} B_E^+[1/t_E])^{\varphi_q=1}\to M 
\end{align*}
is injective. As $(B_E^+[1/t_E])^{\varphi_q=1}=B_{e,E}$, we are reduced to show that the map 
\begin{align}\label{D cris M}
    (M\otimes_{B_{e,E}} B_E^+[1/t_E])^{G_K}\otimes_{K_{0,E}} B_E^+[1/t_E]\to M \otimes_{B_{e,E}}B_E^+[1/t_E]
\end{align}
is injective. Upon replacing $B_E^+[1/t_E]$ by $\mathrm{Frac}(B_E^+[1/t_E])$, the result follows readily from the equality $(\mathrm{Frac}(B_E^+[1/t_E]))^{G_K}=K_{0,E}$ in Lemma \ref{Galois invariants}. Note that the isomorphism \ref{D cris M} also shows that $\dim_{K_{0,E}}D_{\mathrm{cris,E}}(M)\leq \mathrm{rank}_{B_{e,E}}M<\infty$ and that the linearization $\varphi_q^*D_{\mathrm{cris},E}(M)\to D_{\mathrm{cris},E}(M)$ is an isomorphism. Indeed, as the source and target have the same (finite) $K_{0,E}$-dimension, it suffices to show that the map is injective, which in turn can be checked after the faithfully flat extension $K_{0,E}\to B_E^+[1/t_E]$. Thus, we see that the functor $M\mapsto D_{\mathrm{cris},E}(M)$ indeed lands in $\vect(K_{0,E})$, and moreover satisfies $\dim_{K_{0,E}}D_{\mathrm{cris},E}(M)\leq \mathrm{rank}_{B_{e,E}}M$. This finishes the proof of (3).

Finally, (1) follows by combining (2) and (3).
\end{proof}
\begin{defn}[{{\cite[Defn. 10.2.13]{FFCurves}}}]\label{crystalline definition}
(1) A representation $M\in \mathrm{Rep}_{B_{e,E}}(G_K)$ is called crystalline if $M\cong V_{\mathrm{cris},E}(D)$ for some $(D,\varphi_q)\in \vect(K_{0,E})$. We denote by $\mathrm{Rep}_{B_{e,E}}^{\mathrm{cris}}(G_K)$ the full subcategory of crystalline objects in $\mathrm{Rep}_{B_{e,E}}(G_K)$.

(2) A $G_K$-equivariant vector bundle $\E$ on $X_E$ or $X_E\setminus \{\infty\}$ is called crystalline if the $B_{e,E}$-representation $H^0(X_E\setminus \{\infty\},\E)$ is crystalline.
\end{defn}
\begin{lem}\label{crystalline iff M trivial}
$M$ is crystalline if and only if the $B_E^+[1/t_E]$-representation $M\otimes_{B_{e,E}}B_{E}^+[1/t_E]$ is trivial.
\end{lem}
\begin{proof}
This is essentially contained in the proof of Proposition \ref{adjoints V cris}. If $M\otimes_{B_{e,E}}B_{E}^+[1/t_E]$ is trivial, i.e. 
\begin{align*}
    (M\otimes_{B_{e,E}} B_E^+[1/t_E])^{G_K}\otimes_{K_{0,E}} B_E^+[1/t_E]\xrightarrow{\sim} M \otimes_{B_{e,E}}B_E^+[1/t_E],
\end{align*}
then $V_{\mathrm{cris},E}(D_{\mathrm{cris},E}(M))\xrightarrow{\sim} M$ by taking $\varphi_q$-invariants. Conversely, if $M\cong V_{\mathrm{cris},E}(D)$, then by \ref{v cris fully}, we have
\begin{align*}
    M\otimes_{B_{e,E}}B_E^+[1/t_E]&\cong V_{\mathrm{cris},E}(D,\varphi_q) \otimes_{B_{e,E}}B_E^+[1/t_E]\\
    & \xleftarrow{\sim} D\otimes_{K_{0,E}}B_E^+[1/t_E]
\end{align*}
is indeed trivial. 
\end{proof}
\begin{remark}\label{Bcris E}
Recall that $B_{\mathrm{cris},E}$ denote Fontaine's crystalline period ring defined using $E$ and $\tau_0: E\hookrightarrow K\subseteq C$. As in the case $E=\mathbf{Q}_p$ (\cite[Proposition 1.10.12]{FFCurves}), one can check that $B_E^+=\cap_{n\geq 0}\varphi_q^n(B_{\mathrm{cris},E}^+)$ (resp. $B_E^+[1/t_E]=\cap_{n\geq 0}\varphi_q^n(B_{\mathrm{cris},E})$) is the maximal subring of $B_{\mathrm{cris},E}^+$ (resp. $B_{\mathrm{cris},E}$) over which Frobenius is an automorphism. It follows that $D_{\mathrm{cris},E}(M)$ can also be computed as $(M\otimes_{B_{e,E}}B_{\mathrm{cris},E})^{G_K}$ and that $M$ is crystalline if and only if $M\otimes_{B_{e,E}}B_{\mathrm{cris},E}$ is trivial as a $B_{\mathrm{cris},E}$-representation. Indeed, as $\varphi_q$ is an automorphism on $D:=(M\otimes_{B_{e,E}}B_{\mathrm{cris},E})^{G_K}$, we have $D\subseteq \cap_{n\geq 0}\varphi_q^n(M\otimes_{B_{e,E}}B_{\mathrm{cris},E})=M\otimes_{B_{e,E}}\cap_{n\geq 0}\varphi_q^n(B_{\mathrm{cris},E})=M\otimes_{B_{e,E}}B_{E}^+[1/t_E]$, whence $D=D_{\mathrm{cris},E}(M)$.
\end{remark}
\begin{remark}
We have seen that $M$ is crystalline precisely when the natural injective map $V_{\mathrm{cris},E}(D_{\mathrm{cris},E}(M))\hookrightarrow M$ is an isomorphism. In the case $E=\mathbf{Q}_p$, it in fact suffices to require that the source and target have the same $B_{e,E}$-rank, i.e. $\dim_{K_{0,E}} D_{\mathrm{cris},E}(M)=\mathrm{rank}_{B_{e,E}} M$. Indeed, in this case, $\infty$ is the unique closed point in $X_{\mathbf{Q}_p}$ with finite $G_K$-orbit, so any $G_K$-equivariant coherent sheaf on $X_{\mathbf{Q}_p}\setminus \{\infty\}$ must be torsion-free (as its torsion part must have empty support), and hence a vector bundle. Thus, if $\dim_{K_{0,E}} D_{\mathrm{cris},E}(M)=\mathrm{rank}_{B_{e,E}} M$, then the torsion $B_{e,E}$-module $\mathrm{coker}(V_{\mathrm{cris},E}(D_{\mathrm{cris},E}(M))\hookrightarrow M)$ must be zero, as claimed. On the other hand, this does not seem to be enough if $E\ne \mathbf{Q}_p$ since in general there can be more closed points in $X_E$ with finite $G_K$-orbit (for instance, if $K$ contains the Galois closure of $E$ in $\overline{K}$, then $G_K$ fixes $\infty_{\tau}$ for all $\tau: E\hookrightarrow \overline{K}$). (See, however, the proof of Proposition \ref{equivalence E crystalline} below.) This is also related to the fact that the ring $B_E^+[1/t_E]$ is not $(E,G_K)$-regular (as opposed to the case $E=\mathbf{Q}_p$, cf. \cite[Corollaire 10.2.8]{FFCurves}): the line $Et$ is $G_K$-stable yet $t\notin (B_E^+[1/t_E])^\times$. 
\end{remark}
\begin{lem}\label{crystalline stable tensor}
\emph{(1)} The functor $D_{\mathrm{cris},E}$ defines an equivalence $\mathrm{Rep}_{B_{e,E}}^{\mathrm{cris}}(G_K)\xrightarrow{\sim}\vect(K_{0,E})$ with quasi-inverse $V_{\mathrm{cris},E}$.

\emph{(2)} Both $D_{\mathrm{cris},E}$ and $V_{\mathrm{cris},E}$ are exact. 

\emph{(3)} The category $\mathrm{Rep}_{B_{e,E}}^{\mathrm{cris}}(G_K)$ is stable under direct summand subquotients, tensor products, and duals. Moreover, $D_{\mathrm{cris},E}$ naturally respects these operations.
\end{lem}
\begin{proof}
(1) follows immediately from definition and Proposition \ref{adjoints V cris}.

(2) We will show that $V_{\mathrm{cris},E}$ is exact (the argument for $D_{\mathrm{cris},E}$ being analogous). Let $0\to D_1\to D_2\to D_3\to 0$ be an exact sequence in of isocrsyals over $K_{0,E}$. As the inclusion $B_{e,E}\hookrightarrow B_E^+[1/t_E]$ is faithfully flat\footnote{As $B_{e,E}$ is a PID and $B_E^+[1/t_E]$ is a domain, the map is flat. It remains to show that $\m B_E^+[1/t_E]\ne (1)$ for each $\m\in \mathrm{Max}(B_{e,E})$. By \cite[Th\' eor\` eme 6.5.2]{FFCurves}, such $\m$ is generated by $t'/t_E$ for some $t'\in (B_E^+)^{\varphi_q=\pi}\setminus Et_E$. Now $t'/t$ is a not a unit in $B_E^+[1/t_E]$ as otherwise it would be already a unit in $(B_E^+[1/t_E])^{\varphi_q=1}=B_{e,E}$.}, it suffices to show that the induced sequence after applying $V_{\mathrm{cris},E}(\cdot)\otimes_{B_{e,E}}B_E^+[1/t_E]$ is exact. We are now done because this functor is naturally identified with $(\cdot)\otimes_{K_{0,E}}B_E^+[1/t_E]$.

(3) Let $0\to M_1\to M_2\to M_3\to 0$ be an exact sequence in $\mathrm{Rep}_{B_{e,E}}(G_K)$ with $M_2$ crystalline. Consider the commutative diagram 
\begin{displaymath}
\begin{tikzcd}
0 \ar[r] & V_{\mathrm{cris},E}(D_{\mathrm{cris},E}(M_1))\ar[d,hook]\ar[r] &  V_{\mathrm{cris},E}(D_{\mathrm{cris},E}(M_2))\ar[d,"\simeq"]\ar[r] &  V_{\mathrm{cris},E}(D_{\mathrm{cris},E}(M_3))\ar[d,hook] \\ 
0 \ar[r] & M_1 \ar[r] & M_2\ar[r] & M_3\ar[r] & 0
\end{tikzcd}
\end{displaymath}
As the middle vertical arrow is an isomorphism, a simple diagram chasing shows that the two outer maps are also isomorphisms, i.e. $M_1$ and $M_2$ are crystalline, as wanted.

The other claims can be proved e.g. using Lemma \ref{crystalline iff M trivial}.
\end{proof}
Recall that we have a natural functor $(D,\varphi_q)\mapsto \E(D,\varphi_q)$ from $\vect(K_{0,E})$ to the category of $G_K$-equivariant vector bundles on $X_E$, where $\E(D,\varphi_q)$ is the $\O_{X_E}$-module associated to the graded module 
\begin{displaymath}
    \bigoplus_{n\geq 0}(D\otimes_{K_{0,E}}B_E^+)^{\varphi_q=\pi^n}.
\end{displaymath}
By Beauville--Laszlo glueing (applied to the locus ${\infty_{\tau_0}}\hookrightarrow X_E$), the datum of a $G_K$-equivariant vector bundle on $X_E$ is equivalent to the data of a triple $(M_e,M_{\mathrm{dR}}^+,u)$ where $M_e\in \mathrm{Rep}_{B_{e,E}}(G_K), M_{\mathrm{dR}}^+\in \mathrm{Rep}_{B_{\mathrm{dR}}^+}(G_K)$, and $u$ is a $G_K$-equivariant isomorphism $M_e\otimes B_{\mathrm{dR}}\xrightarrow{\sim}M_{\mathrm{dR}}^+[1/t_E]$. In terms of this description, $\E(D,\varphi_q)$ corresponds to the triple $(V_{\mathrm{cris},E}(D),D_K\otimes_{K}B_{\mathrm{dR}}^+,\iota)$ (with $D_K:=D\otimes_{K_{0,E}}K$, and $\iota$ being the natural isomorphism). In particular, by definition, a $G_K$-equivariant vector bundle $\E$ on $X_E$ is crystalline if and only if there exists $(D,\varphi_q)$ so that there is a $G_K$-equivariant isomorphism $\E|_{X_E\setminus\{\infty_{\tau_0}\}}\cong \E(D,\varphi_q)|_{X_E\setminus\{\infty_{\tau_0}\}}$.
\begin{lem}\label{flat Bdr}
Let $V$ be a continuous semilinear representation of $G_K$ on a finite free ${B}_{\mathrm{dR}}^+$-module. Then $V$ is trivial if and only if $V\otimes_{{B}_{\mathrm{dR}}^+}C$ is trivial as a $C$-semilinear representation of ${G_K}.$ 
\end{lem}
\begin{proof}
    See \cite[Proposition 2.18 (1)]{hengdu}. (The proof of \textit{loc. cit. } uses the usual $B_{\mathrm{dR}}^+$ (and the cyclotomic period $t$), but we have seen that the natural map $\widehat{(\O_{X_{\mathbf{Q}}})}_{\infty}\to \widehat{(\O_{X_E})}_{\tau_0}$ is a $G_K$-equivariant isomorphism.)
\end{proof}
\begin{lem}\label{crystalline E0 Dcris}
Let $V\in \mathrm{Rep}_E(G_K)$ be an $E$-representation of $G_K$. Then $V$ is crystalline if and only if $\dim_{K_{0,E}}(V\otimes_{E_0}B_{\mathrm{cris}})^{G_K}=\dim_E V$.
\end{lem}
\begin{proof}
    We have 
    \begin{align*}
        D_{\mathrm{cris}}(V) = (V\otimes_{\mathbf{Q}_p}B_{\mathrm{cris}})^{G_K} & =(V\otimes_{E_0}E_0\otimes_{\mathbf{Q}_p}B_{\mathrm{cris}})^{G_K}\\
        & =\bigoplus_{0\leq i\leq f-1}(V\otimes_{E_0,\varphi_p^i}B_{\mathrm{cris}})^{G_K},
    \end{align*}
    For each $i$, we have $\dim_{E\otimes_{E_0,\varphi_p^i}K_0}(V\otimes_{E_0,\sigma_p^i}B_{\mathrm{cris}})^{G_K}\leq \dim_EV$ with equality if and only if the $E\otimes_{E_0,\varphi_p^i}B_{\mathrm{cris}}$-representation $V\otimes_{E_0,\varphi_p^i}B_{\mathrm{cris}}$ is trivial\footnote{This follows from the usual property of admissible representations, and the fact that $E\otimes_{E_0,\varphi_p^i}B_{\mathrm{cris}}$ is $(E,G_K)$-regular, which in turn can be proved in exactly the same way as in the case $E=\mathbf{Q}_p$.}. As the latter can be obtained from $V\otimes_{E_0}B_{\mathrm{cris}}$ by extending scalars along the ($G_K$-equivariant) map $\varphi_p^i: B_{\mathrm{cris}}\to B_{\mathrm{cris}}$, it in fact suffices to require that $V\otimes_{E_0}B_{\mathrm{cris}}$ is trivial. The lemma now follows by counting dimensions. 
\end{proof}
We can now give a geometric interpretation of the notion of $E$-crystalline representations of Kisin--Ren in terms of vector bundles on the Fargues--Fontaine curve.
\begin{prop}\label{equivalence E crystalline}
Let $V\in \mathrm{Rep}_E(G_K)$. Then the following are equivalent:
\begin{itemize}
    \item[\emph{(1)}] The $G_K$-equivariant vector bundle $V\otimes_E \O_X$ is crystalline in the sense of Definition \ref{crystalline definition}. 
    \item[\emph{(2)}] $V$ is $E$-crystalline.
\end{itemize}
\end{prop}
\begin{example}
Let $V=E(1)$ denotes the $E$-representation of $G_K$ given by the Lubin--Tate character $\chi_{\mathrm{LT}}: G_K\to \O_E^\times$ associated to the uniformizer $\pi$ (and the embedding $\tau_0: E\hookrightarrow K$). Then $V$ satisfies condition (1) of Proposition \ref{equivalence E crystalline}. Indeed, $V\otimes_E B_E^+[1/t_E]$ has a $G_K$-invariant basis given by $v\otimes t_E^{-1}$ where $v$ is an $E$-basis in $V$. More generally, this holds for the $p$-adic Tate module of any $\pi$-divisible $\O_E$-module over $\O_K$; see Lemma \ref{pi divisible and analytic} (the previous example being the case of the Lubin--Tate formal $\O_E$-module associated to $\pi$).
\end{example}
\begin{proof}[Proof of Proposition \ref{equivalence E crystalline}]
Assume (1). By Lemma \ref{crystalline iff M trivial}, $V\otimes_E B_E^+[1/t_E]$ is trivial as a $B_E^+[1/t_E]$-representation of $G_K$. As $t_E$ is invertible in $\widehat{(X_E)}_{\infty_{\tau}}$ for each $\tau \ne \tau_0$ (as $V^+(t_E)={\infty_{\tau_0}}$), by extending scalars, $\bigoplus_{\tau\ne \tau_0}V\otimes_{E,\tau}B_{\mathrm{dR}}^+$ is also trivial (as a $B_{\mathrm{dR}}^+$-representation), whence the same is true for $\oplus_{\tau\ne \tau_0}V\otimes_{E,\tau}C$.

Moreover, by extending scalars along $B_E^+[1/t_E]\hookrightarrow B_{\mathrm{cris}}\otimes_{E_0}E$\footnote{This inclusion is defined as follows. By \cite[Lem. 9.17]{Col02}, $t_E$ divides $t$ in $B_{\max}^+\otimes_{E_0}E$, and so we have $B_E^+[1/t_E]=(B_{\mathbf{Q}_p}^+\otimes_{E_0}E)[1/t_E]\subseteq B_{\max}^+[1/t]\otimes_{E_0}E=B_{\max}\otimes_{E_0}E$. As $\varphi_p(B_{\max})\subseteq B_{\mathrm{cris}}$ and $\varphi_q$ is an automorphism on $B_E^+[1/t_E]$, we also have $B_E^+[1/t_E]\subseteq B_{\mathrm{cris}}\otimes_{E_0}E$, as wanted.}, $V\otimes_{E_0}B_{\mathrm{cris}}$ is also trivial as a $B_{\mathrm{cris}}\otimes_{E_0}E$-representation, and hence taking $G_K$-invariants yields $\dim_{K_{0,E}}(V\otimes_{E_0}B_{\mathrm{cris}})^{G_K}=\dim_EV$. By Lemma \ref{crystalline E0 Dcris}, $V$ is crystalline, as wanted.

Conversely, assume (2) holds. Let $D:=D_{\mathrm{cris},E}(V)=(V\otimes_E B_E^+[1/t_E])^{G_K}$. We need to show that the natural inclusion 
\begin{align}\label{inclusion same rank}
    V_{\mathrm{cris},E}(D)=(D\otimes_{K_{0,E}}B_E^+[1/t_E])^{\varphi_q=1}\hookrightarrow V\otimes_{E} B_{e,E}
\end{align}
is an isomorphism. We first show that the source and target have the same rank, i.e. $\dim_E V=\dim D$. As $B_E^+[1/t_E]\subseteq B_{\mathrm{cris}}\otimes_{E_0}E$, we always have $D\subseteq (V\otimes_{E_0}B_{\mathrm{cris}})^{G_K}$, and the latter has dimension $\dim_E V$ by Lemma \ref{crystalline E0 Dcris}. We will show that $D=(V\otimes_{E_0} B_{\mathrm{cris}})^{G_K}$. As $\mathrm{Rep}_{B_{e,E}}^{\mathrm{cris}}(G_K)$ is stable under tensor products by Lemma \ref{crystalline stable tensor}, by replacing $V$ with $V\otimes_E E(n)$ for $n\ll 0$, we may assume that the Hodge--Tate weights of $V$ are all non-negative. In particular, we have 
\begin{displaymath}
 (V\otimes_{E_0}B_{\mathrm{cris}})^{G_K} =(V\otimes_{E_0}B_{\mathrm{cris}}^+)^{G_K}.  
\end{displaymath}
Moreover, as $\varphi_q$ is an automorphism on $(V\otimes_{E_0}B_{\mathrm{cris}}^+)^{G_K}$, we deduce that 
\begin{align*}
    (V\otimes_{E_0}B_{\mathrm{cris}}^+)^{G_K} &=(V\otimes_{E_0}\cap_{n\geq 0}\varphi_q^n(B_{\mathrm{cris}}^+))^{G_K}\\
    & =(V\otimes_{E_0}B_{\mathbf{Q}_p}^+)^{G_K}\\
    &=(V\otimes_{E}(E\otimes_{E_0}B_{\mathbf{Q}_p}^+))^{G_K}\\
    &=(V\otimes_{E}B_{{E}}^+)^{G_K}\subseteq D,
\end{align*}
as wanted. (For the second equality, see e.g. \cite[\textsection 1.10]{FFCurves}.) 

Next, we claim that the induced map 
\begin{align}\label{completed stalks}
    \bigoplus_{\tau\ne \tau_0}D\otimes_{K_{0,E},\tau}B_{\mathrm{dR}}^+\hookrightarrow \bigoplus_{\tau \ne \tau_0}V\otimes_{E,\tau}B_{\mathrm{dR}}^+ 
\end{align}
on completed stalks is an isomorphism. As $V$ is $E$-crystalline, both the source and target are trivial as a $B_{\mathrm{dR}}^+$-representation by Lemma \ref{flat Bdr}. As any such representation $W$ satisfies $W=(W[1/t])^{G_K}\otimes_K B_{\mathrm{dR}}^+$, it suffices to observe that \ref{completed stalks} becomes an isomorphism after taking $\otimes B_{\mathrm{dR}}$ (being an injection between $B_{\mathrm{dR}}$-vector spaces of the same (finite) dimension). 

Thus, the cokernel of \ref{inclusion same rank} is a torsion equivariant coherent sheaf $\F$ on $X_E\setminus \{\infty_{\tau_0}\}$ satisfying $\widehat{\F}_{\infty_{\tau}}=0$ for all $\tau\ne \tau_0$. As $\{\infty_{\tau}\}_{\tau}=\pi^{-1}(\infty)$ is precisely the set of closed points in $X_E$ with finite $G_K$-orbit, $\F$ must be supported on the set $\{\infty_{\tau}\}_{\tau\ne \tau_0}$, and hence must be zero, as claimed.
\end{proof}
\begin{remark}\label{analytic crys natural notion}
Combining with Remark \ref{Bcris E}, we see that an object $V\in \mathrm{Rep}_E(G_K)$ is $E$-crystalline if and only if $V\otimes_E B_{\mathrm{cris},E}$ is trivial as a $B_{\mathrm{cris},E}$-representation. Thus, the notion of $E$-crystalline representations is in some sense indeed a natural extension of the usual notion for $\mathbf{Q}_p$-representations.  
\end{remark}
\begin{lem}\label{filtration and lattice}
Let $V$ be a finite dimensional $K$-vector space. Then the association $\mathrm{Fil}^\bullet V\mapsto \mathrm{Fil}^0(V\otimes_K B_{\mathrm{dR}})$ gives a bijection between the set of (finite, separated, exhausted) descreasing filtrations on $V$, and the set of $G_K$-equivariant $B_{\mathrm{dR}}^+$-lattices in $V\otimes_K B_{\mathrm{dR}}$. The inverse bijection is given by $W\mapsto (t_E^\bullet W)^{G_K}$.
\end{lem}
\begin{proof}
    See \cite[Proposition 10.4.3]{FFCurves}.
\end{proof}
Combining the above lemma with Beauville--Lazlo's glueing theorem, we deduce the following result.
\begin{lem}
The functor 
\begin{align*}
\mf(K) & \xrightarrow{\sim} \mathrm{Fib}^{G_K,\mathrm{cris}}(X_E)\\
    (D,\varphi_q,\mathrm{Fil}^\bullet D_K) &\mapsto \E(D,\varphi_q,\mathrm{Fil}^\bullet D_K)
\end{align*}
defines an equivalence onto the categories of crystalline $G_K$-equivariant vector bundles on $X_E$. Here $\E(D,\varphi_q,\mathrm{Fil}^\bullet D_K)$ is the modification of $\E(D,\varphi_q)$ at $\infty_{\tau_0}$, defined using the $G_K$-stable $B_{\mathrm{dR}}^+$-lattice $\mathrm{Fil}^\bullet(D_K\otimes_K B_{\mathrm{dR}})$.
\end{lem}
\subsection{``Weakly admissible implies admissible''}
In this subsection, we finish the proof that $E$-crystalline representations are equivalent to weakly admissible filtered isocrystals over $K$ (Theorem \ref{weakly admissible vs crys rep}).

We begin by recalling the notion of weak admissibility for filtered $\varphi_q$-modules. Namely, for a 1-dimensional object $D$ in $\mf(K)$, we pick a basis vector $v\in D$ and let $t_N(D):=v_{\pi}(\alpha)$ where $\alpha\in (K_{0,E})^\times$ is such that $v_q(v)=\alpha v$. We let $t_H(D_K)$ (or more precisely, $t_H(\mathrm{Fil}^\bullet  D_K)$) be the unique integer $i\in\mathbf{Z}$ such that $\mathrm{Fil}^i D_K=D_K$ and $\mathrm{Fil}^{i+1}D_K=0$. For a general $D$, we define $t_H(D_K):=t_H(\det(D_K))$ and $t_N(D):=t_N(\det(D))$. We say that an object $D$ in $\mf(K)$ is weakly admissible if $t_H(D)=t_N(D)$ and $t_H(D')\leq t_N(D')$ for all subobjects $D'\subseteq D$. As in the case $E=\mathbf{Q}_p$, the degree function and the rank function
\begin{align*}
    \deg: (D,\varphi_q,\mathrm{Fil}^\bullet D_K) &\mapsto t_H(D_K)-t_N(D,\varphi_q),\\
    \mathrm{rank}: (D,\varphi_q,\mathrm{Fil}^\bullet D_K)&\mapsto \mathrm{rank}(D,\varphi_q)
\end{align*}
make $\mf(K)$ into a slope category with slope function $\mu:=\deg/\mathrm{rank}$. In particular, each object in $\mf(K)$ admits a unique Harder--Narasimhan filtration, and the resulting abelian subcategory of semistable objects of slope $0$ is precisely formed by those weakly admissible objects in the preceding sense.

For ease of notation, in what follows we will simply write $D$ for a filtered $\varphi_q$-module over $K$, and $\E(D)$ for $\E(D,\varphi_q,\mathrm{Fil}^\bullet)$.

\begin{prop}[{\cite[Proposition 10.5.6]{FFCurves}}]\label{admissible slopes 0} Let $D$ be a filtered $\varphi_q$-module over $K$.

    \emph{(1)} We have $\mathrm{rank}(D)=\mathrm{rank}(\E(D)), \deg(D)=\deg(\E(D))$, and $\mu(D)=\mu(\E(D))$.

    \emph{(2)} If $0=D_0\subsetneq \ldots \subsetneq D_r=D$ is the Harder--Narasimhan filtration of $D$, then that of $\E(D,\varphi_q,\mathrm{Fil}^\bullet D_K)$ is given by 
    \begin{displaymath}
        0=\E(D_0)\subsetneq \ldots\subsetneq \E(D_r)=\E(D).
    \end{displaymath}
    In particular, $D$ is weakly admissible if and only if $\E(D)$ is semistable of slope $0$.
\end{prop}
\begin{proof}
    (1) We have seen in the proof of Proposition \ref{adjoints V cris} that the functor $(D,\varphi_q)\mapsto V_{\mathrm{cris}}(D,\varphi_q)$ is rank-preserving. Thus $\mathrm{rank}(\E(D))=\mathrm{rank}(\E(D,\varphi_q))=\mathrm{rank}(D)$. It remains to show $\det(D)=\deg(\E(D))$. As $\E(D)$ is defined as the modification of $\E(D,\varphi_q)$ at $\infty_{\tau_0}$ using the lattice $\mathrm{Fil}^0(D_K\otimes_K B_{\mathrm{dR}})$, we have
    \begin{displaymath}
        \deg(\E(D))=\deg(\E(D,\varphi_q))-[D_K\otimes_K B_{\mathrm{dR}}^+:\mathrm{Fil}^0(D_K\otimes_K B_{\mathrm{dR}})]\footnote{For an effective modification $0\to \E'\to \E\to \F\to 0$ (so that $\F$ is a skyscraper sheaf, supported at $\infty_{\tau_0}$), this follows from additivity of degree: $\deg(\E)=\deg(E')+\mathrm{length}(\F)=\deg(\E)+[\widehat{(\E)}_{\infty_{\tau_0}}:\widehat{(\E')}_{\infty_{\tau_0}}]$. In general, we can choose $n\gg 0 $ so that $\widehat{(\E')}_{\infty_{\tau_0}}\subseteq t_E^{-n}\widehat{(\E)}_{\infty_{\tau_0}}$, and hence reduce to effective case.}
    \end{displaymath}
    Now $\deg(\E(D,\varphi_q))=-t_N(D,\varphi_q)$ (recall that if $(D,\varphi_q)=\O(\lambda)$, then $\E(D,\varphi_q)=\O(-\lambda)$), while it follows easily by choosing a splitting of the filtration on $D_K$ that
    \begin{displaymath}
        [D_K\otimes_K B_{\mathrm{dR}}^+:\mathrm{Fil}^0(D_K\otimes_K B_{\mathrm{dR}})]=-t_H(\mathrm{Fil}^\bullet D_K),
    \end{displaymath}
    as desired.

    (2) We follow the proof of \cite[Proposition 10.5.6]{FFCurves}. Observe firstly that by uniqueness, the Harder--Narasimhan filtration of $\E(D)$ is $G_K$-equivariant. Thus, by part (1) and definition of semistability, it suffices to show that if $\E'\subseteq \E(D)$ is a $G_K$-stable subbundle, then $\E'=\E(D')$ for a (necessarily unique) subobject $D'\subseteq D$ (as a filtered $\varphi_q$-module). As $\E'$ is a subbundle, $\E'|_{X_E\setminus\{\infty_{\tau_0}\}}$ is in particular is a Galois stable direct summand of $V_{\mathrm{cris},E}(D)$, and hence crystalline by Lemma \ref{crystalline stable tensor}. Thus, $\E'=\E(D',\varphi_q)$ on $X_E\setminus\{\infty_{\tau_0}\}$ for some $\varphi_q$-module $D'\subseteq D$. The lattice $\widehat{(\E')}_{\infty_{\tau_0}}$ determines a filtration on $D'_K$, which we claim is simply the one inherited from $D_K$. Indeed, if $\E''$ denotes the equivariant vector bundle given by this latter filtration, then it follows from the (explicit) bijection in Lemma \ref{filtration and lattice} that $\E'\subseteq \E''$. As $\E'$ is a subbundle in $\E$ (hence in $\E''$) and $\mathrm{rank}(\E')=\mathrm{rank}(\E'')=\dim D'$, we must have $\E'=\E''$. Thus, we see that $\E'=\E(D')$ for a subobject $D'\subseteq D$, as claimed.
\end{proof}
After the classification of vector bundles (\cite[Th\' eor\` eme 8.2.10]{FFCurves}), any vector bundle $\E$ on $X_E$ is of the form
\begin{displaymath}
    \E\simeq\O(\lambda_1)\oplus \ldots \oplus \O(\lambda_n)
\end{displaymath}
for a unique tuple $(\lambda_1\geq \ldots \geq \lambda_n)$ of rational numbers. As 
\begin{displaymath}
    \dim_E H^0(X,\O(\lambda))=\begin{cases} 0 \quad\text{if $\lambda<0$},\\
    1 \quad\text{if $\lambda=0$},\\
    \infty\quad\text{if $\lambda>0$},\end{cases}
\end{displaymath}
$\E$ is semistable of slope $0$ if and only if $\dim_E H^0(X_E,\E)=\mathrm{rank}(\E)$. Combining with Proposition \ref{admissible slopes 0}, we see that a filtered $\varphi_q$-module $D$ over $K$ is weakly admissible if and only if $\dim_{K_{0,E}}D=\dim_E V_E(D)$, where $V_E(D):=H^0(X_E,\E(D))=(D\otimes_{K_{0,E}}B_E^+[1/t_E])^{\varphi_q=1}\cap \mathrm{Fil}^0(D_K\otimes_K B_{\mathrm{dR}})$. 

Motivated by the case $E=\mathbf{Q}_p$, we next proceed to show that the functor $D\mapsto V_E(D)$ defines an equivalence between the category of weakly admissible filtered $\varphi_q$-modules over $K$, and the category of $E$-crystalline representations of $G_K$. 

Let $V\in\mathrm{Rep}_E(G_K)$. Define 
\begin{displaymath}
    D_{\mathrm{cris},E}(V):=(V\otimes_E B_E^+[1/t_E])^{G_K}.
\end{displaymath}
Of course, this is nothing but $D_{\mathrm{cris},E}(M)$ where $M:=V\otimes_E B_{e,E}$. In particular, we have seen that $D_{\mathrm{cris},E}(V)$ is naturally a $\varphi_q$-module over $K_{0,E}$, of dimension $\leq \dim_E V$. Via the natural inclusion $B_E^+[1/t_E]\otimes_{K_{0,E}}K\hookrightarrow B_{\mathrm{dR}}$, we can endow 
\begin{displaymath}
    D_{\mathrm{cris},E}(V)\otimes_{K_{0,E}}K\hookrightarrow (V\otimes_{E}B_{\mathrm{dR}})^{G_K}
\end{displaymath}
with the subspace filtration from $V\otimes_E B_{\mathrm{dR}}$. In this way, $D:=D_{\mathrm{cris},E}(V)$ is naturally a filtered $\varphi_q$-module over $K$.

\begin{thm}\label{weakly admissible vs crys rep}
The functor 
\begin{align*}
    D_{\mathrm{cris},E}: \mathrm{Rep}_{E}^{\mathrm{cris}}(G_K)\to \mf(K)
\end{align*}
is fully faithful. Moreover, the essential image is precisely the subcategory of weakly admissible objects.
\end{thm}
\begin{proof}
Recall that $B_{e,E}\cap B_{\mathrm{dR}}^+=E$ (as follows from the fundamental exact sequence $0\to E\to B_{e,E}\to B_{\mathrm{dR}}/B_{\mathrm{dR}}^+\to 0$). The first statement follows rather formally from this. More precisely, for each $V$ is the source, we have 
\begin{align*}
    V =(V\otimes_E B_E^+[1/t_E])^{\varphi_q=1}\cap \mathrm{Fil}^0(V\otimes_E B_{\mathrm{dR}})
=  V_E(D_{\mathrm{cris},E}(V)).
\end{align*}
Taking $G_K$-invariants yields, $V^{G_K}=\mathrm{Fil}^0(D_{\mathrm{cris},E}(V)^{\varphi_q=1})$. Using a suitable internal Hom, this implies full faithfulness of $D_{\mathrm{cris},E}$. We remark also that $V_E$ is a quasi-inverse on the essential image of $D_{\mathrm{cris},E}$. We next show that $D:=D_{\mathrm{cris},E}(V)$ is weakly admissible. This follows from the equality $\dim_E V_E(D)=\dim_E V=\dim_{K_{0,E}}D$, Proposition \ref{admissible slopes 0}, and the classification of vector bundles on $X_E$.

It remains to show that if $D$ is a weakly admissible filtered $\varphi_q$-module over $K$, then $V:=V_E(D)$ is $E$-crystalline, and $D_{\mathrm{cris},E}(V)\simeq D$ as filtered $\varphi_q$-modules. We will follow the proof of \cite[Proposition 4.5]{ColmezFontaine}. Let $C_E$ denote the fraction field of $B_E^+[1/t_E]$. As $(C_E)^{G_K}=K_{0,E}$ by Lemma \ref{Galois invariants}, by \cite[Lemme 4.6]{ColmezFontaine}, there exists a (necessarily unique) $K_{0,E}$-vector space $D'\subseteq D$ such that $D'\otimes_{K_{0,E}}C_E$ equals the $C_E$-subspace of $D\otimes_{K_{0,E}}C_E$ generated by $V$. As $V$ is fixed by $\varphi_q$, $D'$ is $\varphi_q$-stable, and hence naturally a subobject of $D$ (as a filtered $\varphi_q$-module). Moreover, as $V\subseteq D'\otimes C_E$ and $V\subseteq D\otimes B_E^+[1/t_E]$, we have $V\subseteq D'\otimes B_E^+[1/t_E]$, whence $V=V_E(D')$. Let $d_1,\ldots, d_r$ be a $K_{0,E}$-basis of $D'$. Choose also $v_1,\ldots,v_r\in V$ which spans $D'\otimes_{K_{0,E}} C_E$ over $C_E$. For each $i$, write $v_i=\sum_{j}b_{ij}d_j$ for some $b_{ij}\in B_E^+[1/t_E]$. Then $b:=\det(b_{ij})$ is nonzero, and so 
\begin{displaymath}
    w:=v_1\wedge\ldots \wedge v_r=b(d_1\wedge \ldots \wedge d_r)
\end{displaymath}
is a nonzero element in $W:=V_E(\wedge^r D')\subseteq B_E^+[1/t_E]\otimes \wedge^r D'$. As $t_H(D')\leq t_N(D')$ (by weak admissibility of $D$), it follows from Lemma \ref{rank 1 case} below that $t_H(D')=t_N(D'), W=Ew$ and that $b$ is a unit in $B_E^+[1/t_E]$. Thus the natural map $V\otimes_E B_E^+[1/t_E]\to D'\otimes B_E^+[1/t_E]$ is surjective, and so as the source and target are abstractly isomorphic (as $\dim_E V=\dim_E V_E(D')=\dim_{K_{0,E}}D'$ by weak admissibility of $D'$), it is in fact an isomorphism. Thus $V$ is $E$-crystalline and $D_{\mathrm{cris},E}(V)=D'\subseteq D$. Finally as $\dim D'=\dim_E V=\dim D$, we must have $D'=D$. 
\end{proof}
\begin{lem}\label{rank 1 case}
Let $D$ be a filtered $\varphi_q$-module over $K$. Assume $D$ is $1$-dimensional with a basis $d$. Then 
\begin{displaymath}
    \dim_E V_E(D)=\begin{cases} 0\quad \text{if $t_H(D)<t_N(D)$},\\
    1\quad \text{if $t_H(D)=t_N(D)$},\\
    \infty\quad \text{if $t_H(D)>t_N(D)$}.\end{cases}
\end{displaymath}
Moreover, in case $\dim V_E(D)=1$, any basis of $V_E(D)$ is of the form $bd$ for some unit $b\in B_E^+[1/t_E]$.
\end{lem}
\begin{proof}
    Write $\varphi_q(d)=\pi^{t_N(D)}uv$ with $u\in W_{\O_E}(k)^\times$. Choose $x\in W_{\O_E}(\overline{k})^\times$ so that $\varphi_q(x)=ux$. One then checks easily that
    \begin{displaymath}
        V_E(D)=t_E^{-t_H(D)}x^{-1}\mathrm{Fil}^{0}(B_E^+[1/t_E])^{\varphi_q=\pi^{t_H(D)-t_N(D)}}.
    \end{displaymath}
The lemma now follows from this and the fundamental exact sequence.
\end{proof}
\begin{remark}
For a weakly admissible $D$ in $\mf(K)$, let $V_E'(D):=(D\otimes_{K_{0,E}}(B_{\mathrm{cris}}\otimes_{E_0}E))^{\varphi_q=1}\cap \mathrm{Fil}^0(D_K\otimes_K B_{\mathrm{dR}})$. As $B_E^+[1/t_E]\subseteq B_{\mathrm{cris}}\otimes_{E_0}E$, $V_E(D)\subseteq V_E'(D)$. Moreover, by \cite[Proposition (3.3.4)]{KisinRen}, $\dim_E V_E'(D)\leq \dim D$. As $\dim V_E(D)=\dim H^0(X_E,\E(D))=\dim D$ (recall that $\E(D)$ is semistable of slope $0$ by weak admissibility of $D$), $V_E'(D)=V_E(D)$. Thus, the definition of $V_E(D)$ here agrees with the one in \cite[\textsection 3]{KisinRen} (for weakly admissible $D$).
\end{remark}
\medskip
\printbibliography
\addcontentsline{toc}{section}{References}

\textsc{\small LAGA, Universit\' e Paris 13, 99 Avenue Jean-Baptiste Cl\' ement, 93430 Villetaneuse, France}\\
\noindent\textit{Email address}: \href{mailto:dat.pham@math.univ-paris13.fr}{\texttt{dat.pham@math.univ-paris13.fr}}\\

Current address: \textsc{\small CNRS, IMJ-PRG – UMR 7586, Sorbonne Universit\' e, 4 place Jussieu, 75005 Paris, France}\\
\indent\textit{Email address}: \href{mailto:dat.pham@imj-prg.fr}{\texttt{dat.pham@imj-prg.fr}}
\end{document}